\documentclass[reqno,centertags,a4paper,11pt]{amsart}
\usepackage{amssymb}
\usepackage{amsthm}
\usepackage{amsmath}
\usepackage{eucal}
\usepackage{accents}
\usepackage{color, soul, xcolor}
\soulregister\cite{7}
\soulregister\ref{7}
\soulregister\pageref7
\usepackage[colorlinks=true,linkcolor=blue]{hyperref}
\hypersetup{colorlinks,linkcolor={blue},citecolor={red},urlcolor={blue}}  
\usepackage[english]{babel}
\usepackage{graphicx}
\usepackage{subfig}

\newcommand{\R}{\mathbb R}

\newcommand{\Z}{\mathbb Z}
\newcommand{\T}{\mathbb T}

\newcommand{\p}{\partial}

\numberwithin{equation}{section}
\newtheorem{theorem}{Theorem}[section]
\newtheorem{proposition}[theorem]{Proposition}
\newtheorem{remark}[theorem]{Remark}
\newtheorem{lemma}[theorem]{Lemma}
\newtheorem{corollary}[theorem]{Corollary}
\newtheorem{definition}[theorem]{Definition}

\begin{document}
\title[Well-posedeness]{Local well-posedness for a system of modified KdV equations in modulation spaces}

\author{X. Carvajal}
\address{Instituto de Matem\'atica, UFRJ, 21941-909, Rio de Janeiro, RJ, Brazil}
\email{carvajal@im.ufrj.br}
\author{F. Cuba}
\address{Instituto de Matem\'atica, UFRJ, 21941-909, Rio de Janeiro, RJ, Brazil}
\email{fidel.cubab@unmsm.edu.pe }
\author{M. Panthee}
\address{Department of Matematics, University of Campinas (UNICAMP), Campinas, SP, Brazil}
\email{mpanthee@unicamp.br}

\begin{abstract}
 In this work, we consider the initial value problem (IVP) for a system of modified Korteweg-de Vries (mKdV) equations
 \begin{equation*}
 \begin{cases}
\partial_t v  +  \partial_x^3 v+ \partial_x (v w^2)  = 0, \hspace{0.98 cm} v(x,0)=\psi(x),\\
\partial_t w  +  \alpha \partial_x^3 w+\partial_x (v^2 w)  = 0,\hspace{0.5 cm} w(x,0)=\phi(x).
\end{cases}
\end{equation*}
The main interest is in addressing the well-posedness issues of the IVP when the initial data are considered in the modulation  space $M_s^{2,p}(\R)$, $p\geq 2$. In the case when $0<\alpha\ne 1$, we derive new trilinear estimates in these spaces and prove that the IVP is locally well-posed for data in $M_s^{2,p}(\R)$ whenever $s> \frac14-\frac{1}{p}$ and $p\geq 2$. In deriving the trilinear estimate, the fact that the Fourier supports of the solution components $v$ and $w$  lie on distinct cubic curves, namely $\tau = \xi^3$ and $\tau = \alpha\xi^3$, introduces additional difficulties in handling the resonant case. This makes the analysis substantially different from what one encounters in the single-equation setting.  To overcome the difficulties arising in the resonant case, it was necessary to impose the more restrictive condition  $s> \frac14-\frac{1}{p}$  on the trilinear estimate, rather than the natural threshold $s> \frac14-\frac{3}{2p}$ , which would otherwise yield sharp local well-posedness for $s>-\frac12$  when $p=2$.
\end{abstract}

\maketitle

\section{Introduction}

In  this work, we consider the initial value problem (IVP) for  the following system of the modified Korteweg-de Vries (mKdV) type equations
\begin{equation}\label{PVI_Majda-Biello equation}
	\left\{
	\begin{aligned}
		& \partial_t v  +  \partial_x^3 v+ \partial_x (v w^2)  = 0, \hspace{0.98 cm} v(x,0)=\psi(x),\\
		& \partial_t w  +  \alpha \partial_x^3 w+\partial_x (v^2 w)  = 0,
		\hspace{0.5 cm} w(x,0)=\phi(x),
	\end{aligned}
	\right.
\end{equation}
with given data in modulation spaces (see \eqref{Def_Esp_Modulation} below), where $x,t\in\mathbb{R}$, $0<\alpha\neq 1$, $v=v(x,t)$ and $w=w(x,t)$ are real-valued functions.

For $\alpha=1$ the system (\ref{PVI_Majda-Biello equation}) reduces to a special case of broad class of nonlinear evolution cosidered by Ablowitz, Kaup, Newell and Segur \cite{M_Ablowitz} in the inverse scattering context. In this scenario, the issues of well-posedness, existence, and stability of solitary waves for this system have been extensively investigated in the literature.
\\
Applying the technique developed by Kenig, Ponce and Vega \cite{C_Kening_Ponce}, we can prove that the IVP (\ref{PVI_Majda-Biello equation}) is locally well-posed for given data in $H^s(\mathbb{R})\times H^s(\mathbb{R})$, for $s\geq \frac{1}{4}$ and $\alpha=1$. This approach utilizes the smoothing effect of the linear group, combined with the $L^p_x L^q_t$ Strichartz estimates and maximal function estimates. Tao \cite{Tao_2} demonstrated that this local result can also be established using the restricted norm method in the Fourier transform space $X^{s,b}$, as introduced by Bourgain \cite{Bourgain}. In \cite{Montenegro} Montenegro proved the local well-posedness for $s\geq\frac{1}{4}$ and, using the conservation laws of the system (\ref{PVI_Majda-Biello equation}), also proves global well-posedness for initial data in $H^s(\mathbb{R})\times H^s(\mathbb{R})$, for $s\geq 1$. For further improvement of the global result for $s>\frac14$, we refer to \cite{CP}.\\

For $0<\alpha<1$, the  well-posedness issue for the IVP (\ref{PVI_Majda-Biello equation}) is quite different. Although the technique of Kenig, Ponce and Vega \cite{C_Kening} readily yields the local well-posedness in $H^s(\mathbb{R})\times H^s(\mathbb{R})$ for $s\geq \frac{1}{4}$,  the Fourier transform restriction norm space introduced by Bourgain \cite{Bourgain}, provides a much better result as was shown in our previous work \cite{Carvajal_and_Panthee}. While using Bourgain space, the resonance term coming from different groups associated to the linear part behaves better to get the crucial trilinear estimate used in the local well-posedness result. In fact, we taking idea from Oh \cite{Tadahiro_1}, in \cite{Carvajal_and_Panthee}, we derived the following trilinear estimates for intereacting nonlinearity
\begin{equation}\label{tlint-m1}
\|(vw^2)_x\|_{X^{s,b'}}\lesssim  \|v\|_{X^{s,b}} \|w\|_{X_{\alpha}^{s,b}}^2
\end{equation}
and
\begin{equation}\label{tlint-m2}
 \|(v^2w)_x\|_{X_{\alpha}^{s,b'}}\lesssim  \|v\|_{X^{s,b}}^2 \|w\|_{X_{\alpha}^{s,b}},
\end{equation}
 which hold for $s>-\frac12$ when $0<\alpha<1$, thereby yielding the local well-posedness result in $H^s(\mathbb{R})\times H^s(\mathbb{R})$ for $s> -\frac12$ (For definition of $X^{s,b}$ and $X^{s,b}_\alpha$ spaces, see (\ref{Bourgain spaces adapted})).
 
At this point we mention the work of Oh \cite{Tadahiro_1,Tadahiro_2} on the KdV and the Majda-Biello system introduced in \cite{Majda_Biello}
\begin{equation}\label{Majda-Biello}
	\left\{
	\begin{aligned}
		& \partial_t v+\partial^3_x v+\partial_x(w^2) 
		=  0,
		\hspace{0.98 cm} v(x,0)=\psi(x),\\
		& \partial_t w+\alpha \partial^3_x w+\partial_x(v w) 
		=  0,
		\hspace{0.5 cm} w(x,0)=\phi(x),
	\end{aligned}
	\right.
\end{equation}
where $0< \alpha< 1$. The author in \cite{Tadahiro_1} used the Fourier transform restriction norm method and proved that the IVP (\ref{Majda-Biello}) is locally well posed for data with regularity $s\geq 0$. The main tool in the proof was the validity of the bilinear estimate for the interacting nonlinearity
\begin{equation}\label{Bilinear estimate_introduct}
	\|\partial_x(vw)\|_{X^{s,b'}_\alpha}\lesssim \|v\|_{X^{s,b}} \|w\|_{X^{s,b}_\alpha},
\end{equation}
whenever $s\geq 0$. Recall that the bilinear estimate (\ref{Bilinear estimate_introduct}) in the case $\alpha=1$ holds for $s>-\frac{3}{4}$ (see \cite{C_Kening_2}). In this sense the mKdV system with different groups behaves better than the KdV system with different groups. In the case $\alpha\in (0,1)$ the local result in  \cite{Carvajal_and_Panthee} was further extended in the work \cite{Carvajal_Esquivel_Santos} considering initial data in $H^s(\mathbb{R})\times H^k(\mathbb{R})$,  $k>-\frac{1}{2}$, $s>-\frac{1}{2}$ and $|k-s|\leq\frac{1}{2}$ (see also \cite{Carvajal_3}). 
\\

Recently, the study of the well-posedness of IVPs associated with nonlinear dispersive equations has been studied in function spaces other than the usual  $L^2$ based Sobolev spaces $H^s(\mathbb{R})$, namely, the Fourier-Lebesgue spaces $\mathbf{F}\mathbf{L}^{s,p}(\mathbb{R})$ with the norm
\begin{equation*}
	\|u\|_{\mathbf{F}\mathbf{L}^{s,p}(\mathbb{R})}=\|\langle\xi\rangle^s \widehat{u}(\xi)\|_{L^p},
\end{equation*}
and modulation spaces $M^{r,p}_s(\mathbb{R})$ with the norm given by \eqref{Def_Esp_Modulation}, below. More specifically, we highlight the local well-posedness result for the modified Korteweg-de Vries (mKdV) equation in $\mathbf{FL}^{s,p}(\mathbb{R})$ for $s\geq \frac{1}{2p}$ with $2\leq p<4$ obtained in \cite{Bar:04} and its improvement for the same range of $s$ with $2\leq p<\infty$ in \cite{Grunrock_2}.\\

For further discussion on these results, we refer readers to \cite{Bar:02}, where the authors considered the IVP for the complex-valued mKdV equation in the modulation spaces $M^{2,p}_s(\mathbb{R})$ and proved local well-posedness for $s\geq\frac{1}{4}$ with $2\leq p<\infty$. In \cite{Tadahiro_and_Y_Wang_2} T. Oh and Y. Wang proved that the cubic nonlinear Schr\"odinger (NLS) equation on $\mathbb{R}$:
\begin{equation}\label{one-dimensional cubic NLS}
	\left\{
	\begin{aligned}
		& i \partial_t  u  =  \partial_x^2 u \mp 2|u|^2 u, \hspace{0.5cm} x,t\in\mathbb{R},\\
		& u(x,0)  = u_0(x),
	\end{aligned}
	\right.
\end{equation}
is globally well-posed in $M^{2,p}(\mathbb{R})$ for any $p<\infty$. They also proved that the normalized cubic NLS on $\mathbb{T}$:
\begin{equation}\label{normalized cubic NLS}
	\left\{
	\begin{aligned}
		& i \partial_t  u  =  \partial_x^2 u \mp 2(|u|^2- \frac{1}{\pi}\int_{\mathbb{T}} |u|^2\,dx) u, \hspace{0.5cm} (x,t)\in \mathbb{T}\times\mathbb{R},\\
		& u(x,0)  = u_0(x),
	\end{aligned}
	\right.
\end{equation}
is globally well-posed in $\mathbf{F}\mathbf{L}^p(\mathbb{T})$ for any $1\leq p<\infty$ by introducing a new function space $H\!M^{\theta,p}$ whose norm is given by the $\ell^p$-sum of the modulated $H^{\theta}$-norm of a given function that agrees with the modulation space $M^{2,p}(\R)$ on the real line and Fourier-Lebesgue space $\mathcal{F}L^p(\T)$ on the circle.   Also, in very recent work \cite{Bar:03} the extended nonlinear Schr\"odinger (e-NLS) equation and the higher order nonlinear Schr\"odinger (h-NLS)  equation were considered and proved to be locally well-posed for given data in  the modulation space $M_s^{2,p}(\R)$ respectively for $s>-\frac14$ and $s\geq\frac14$ with $2\leq p<\infty$. Further, recent works \cite{BV-21},  \cite{BV-20}, \cite{CP-24} that respectively deal with  the BBM equation, Boussinesq equation and a higher order water wave model with given initial data in the modulation spaces are worth mentioning.
\\

The main objective of this work is in addressing the well-posedness issues for the IVP (\ref{PVI_Majda-Biello equation}) with given data in the modulation spaces. For $\alpha =1$, it is quite easy to adapt the ideas in \cite{Grunrock_2}, \cite{Grunrock_2} and \cite{Bar:02} and obtain the similar results for the IVP \eqref{PVI_Majda-Biello equation} as well. However, for $\alpha\ne 1$, the scenario changes drastically as described above. Taking in consideration this fact, we assume $\alpha\in (0, 1)$ and  consider the (\ref{PVI_Majda-Biello equation}) with given data in the modulation spaces. With this consideration , we derive new trilinear estimates for the interacting nonlinear terms that hold whenever  $s>\frac{1}{4}-\frac1p$. Consequently, we obtain local well-posedness results for the IVP (\ref{PVI_Majda-Biello equation}) for given data $(v_0,w_0)\in M^{2,p}_s(\mathbb{R})\times M^{2,p}_s(\mathbb{R})$, $s>\frac{1}{4}-\frac1p$ and $2\leq p<\infty$. More precisely, we establish the following local well-posedness result.

\begin{theorem}\label{maintheo}
	Let $\alpha\in (0, 1)$, $s> \frac{1}{4}-\frac{1}{p}$ and $2\leq p <\infty$, then for any $(v_0,w_0)\in M^{2,p}_s (\mathbb{R})\times M^{2,p}_s(\mathbb{R})$, there exist $T=T(\|(v_0,w_0)\|_{M^{2,p}_s\times M^{2,p}_s})>0$ and a unique solution $(v,w)\in X^{s,b}_p(I)\times X^{s,b}_{p,\alpha}(I)$ to the IVP (\ref{PVI_Majda-Biello equation}) in the time interval $I=[0,T]$. Moreover, the solution satisfies the estimate
	\begin{equation*}
		\|(v,w)\|_{X^{s,b}_p(I)\times X^{s,b}_{p,\alpha}(I)}\lesssim \|(v_0,w_0)\|_{M^{2,p}_s\times M^{2,p}_s}.
	\end{equation*}
	where the norm $\|\cdot\|_{X^{s,b}_p(I)}$ and $\|\cdot\|_{X^{s,b}_{p,\alpha}(I)}$ are as defined in \eqref{Bourgain spaces adapted} and \eqref{X^{s,b}_p(I)}.
\end{theorem}

\begin{remark}
As noted in \cite{Carvajal_and_Panthee}, the result stated in Theorem \ref{maintheo} holds for $\alpha>1$ as well. This can be justified by using symmetry of the system and scaling $\tilde{v}(x, t) =v(\alpha^{-\frac13}x, t)$ and $\tilde{w}(x, t) =w(\alpha^{-\frac13}x, t)$, so that $(\tilde{u}, \tilde{v})$ satisfy
 \begin{equation*}
\begin{cases}
\p_t\tilde{v}+\frac{1}{\alpha} \p_x^3\tilde{v} + \p_x(\tilde{v}\tilde{w}^2) =0,\\
\p_t\tilde{w} + \p_x^3\tilde{w} + \p_x(\tilde{v}^2\tilde{w}) =0,
\end{cases}
\end{equation*}
with $0<\frac{1}{\alpha} <1$.
\end{remark}

\begin{remark}
The main tool used to prove Theorem \ref{maintheo} is the trilinear estimate established in Proposition \ref{estimativ bilinear u1u2} below. A key source of difficulty in deriving this estimate arises from the fact that the Fourier supports of the solution components  $v$ and $w$ lie on the distinct cubic curves, namely $\tau = \xi^3$ and $\tau = \alpha\xi^3$, respectively. This structural difference significantly complicates the treatment of the resonant interactions, making the analysis fundamentally more involved than in the case of a single equation, where all components share the same dispersion relation (see the resonant Case D in the proof of trilinear estimate \eqref{trilinear estimate_1}, particularly the Sub-subcases D1.2 and D1.5; see also Lemma~\ref{lemma alpha difert to 0} and Corollary \ref{Pi_mnalpha}). In particular, the mismatch between the cubic phases leads to nontrivial interactions in the resonant regime that cannot be treated with the standard multilinear estimates typically employed for single equations. To overcome these difficulties, it becomes necessary to impose a more restrictive condition on the regularity index, namely $s> \frac14-\frac{1}{p}$, rather than the natural threshold $s> \frac14-\frac{3}{2p}$, which would otherwise suffice to recover the sharp local well-posedness threshold  $s>-\frac12$   in the case $p=2$ obtained in \cite{Carvajal_and_Panthee}. This loss reflects the additional complexity introduced by the interaction of components with differing dispersion relations in the resonant regime.
\end{remark}


\section{Function Spaces and basic estimates}


In this section, we introduce the function spaces, their properties and some basic estimates that will be useful throughout the work. We begin by introducing the modulation spaces, originally defined by H. G. Feichtinger in 1983 \cite{Feich-83} as part of an effort to establish smoothness spaces over locally compact Abelian groups. Notably, Feichtinger also developed a parallel theory to the well-known Besov spaces, and modulation spaces have inspired the broader framework of coorbit spaces \cite{Feich-1, Feich-2, Feich-88}. For a historical account of the development of modulation spaces, we refer to an expository paper by the same author \cite{F-06}. In this work, we will employ the modulation space formulation based on an equivalent norm, as presented in \cite{Bar:02}.
\begin{definition}
	Let $\omega>0$ and $\psi^{\omega} \in \mathbb{S}(\R)$ be such that
	\begin{center}
		$supp (\psi^{\omega})\subset [-\frac1{\omega},\frac1{\omega}]$ \,\, and\,\, $\sum_{n\in\mathbb{Z}}\psi^{\omega} (\frac{\xi}{\omega}-\frac{n}{\omega})=1$.
	\end{center}
For given $n\in \mathbb{Z}$ define
	\begin{equation}\label{Projetor L-Pg}
		\widehat{f_n^{\omega}} (\xi):=\widehat{\Pi_n^{\omega} f}(\xi):= \psi^{\omega} \left(\frac{\xi}{\omega}-\frac{n}{\omega}\right)\widehat{f}(\xi),
	\end{equation}
	and for simplicity of exposition denote by $\psi:=\psi^{1}$, $\psi_n(\xi):=\psi(\xi-n)$, $\Pi_n:=\Pi_n^{1}$, and $f_n:=f_n^{1}$.
	
	Given $1\leq r,p\leq \infty$ and $s\in\mathbb{R}$,  the modulation space $M^{r,p}_s(\mathbb{R})$ is defined as the collection of all tempered distributions  
	$f\in \mathcal{S}'(\mathbb{R})$ \,such that\, $\|f\|_{M^{r,p}_s(\mathbb{R})}<\infty$, where the norm
	\,$M^{r,p}_s(\mathbb{R})$\, is defined by 
	\begin{equation}\label{Def_Esp_Modulation}
		\|f\|_{M^{r,p}_s(\mathbb{R})}:=
		\|\langle n\rangle^s\|\psi_n(D)f\|_{L^r_x(\mathbb{R})}\|_{\ell^{p}_n(\mathbb{Z})},
	\end{equation}
	and $\,\psi_n(D)\,$ is the Fourier multiplier operator with the multiplier
	$\psi_n(\xi).$
\end{definition}

\begin{definition}
	Given dyadic $N\geq 1$, we define  the Littlewood-Paley projector $P_N$ onto the (spatial) frequencies $\{|\xi| \sim N\}$ via
	\begin{equation*}
		\widehat{P_N f}(\xi):=\varphi_N(\xi)\widehat{f}(\xi),
	\end{equation*}
	where $\varphi\in\mathcal{S}(\mathbb{R})$, $\varphi_N(\xi):=\varphi(\frac{\xi}{N})$, $supp\,\varphi \subseteq \{1\leq |\xi|\leq 4\}$ and $\varphi=1$ in $\{2\leq |\xi|\leq 3\}$.\\
	Also, for given $n\in \mathbb{Z}$ we define
	\begin{equation}\label{Projetor L-P}
		\widehat{\Pi_n f}(\xi):= \psi_n(\xi)\widehat{f}(\xi),
	\end{equation}
	where $\psi_n(\xi)=\psi(\xi-n)$, in particular if $\xi\in supp\,{\psi}$, then $\xi=n+O(1)$.
\end{definition}
\begin{lemma}\text{({Bernstein inequalities})}
	For each $1\leq q\leq p\leq \infty$, one has
	\begin{equation}\label{Desigual_Bernstein}
		\|P_N f\|_{L^p_x} 
		\lesssim N^{\frac{1}{q}-\frac{1}{p}} \|f\|_{L^q_x},			
	\end{equation}
	\begin{equation}\label{Desigual_Bernstein_1}
		\|\Pi_n f\|_{L^p_x}\lesssim \|f\|_{L^q_x}.
	\end{equation}
\end{lemma}


\begin{definition}\label{Defi_esp_modula}
	The Bourgain type spaces $X^{s,b}_{p,\alpha}$ adapted to the modulation spaces $M^{2,p}_s(\mathbb{R})$ are defined  with the norm given by 
	\begin{equation}\label{Bourgain spaces adapted}
		\|f\|_{X^{s,b}_{p,\alpha}}:=\left(\sum_{n\in \mathbb{Z}} 
		\langle n \rangle^{sp}
		\|\langle\tau-\alpha\xi^3\rangle^b \widehat{f}(\xi, \tau)\|^p_{L^2_{\tau,\xi}(\mathbb{R}\times[n,n+1])}\right)^{\frac{1}{p}}
		\sim \|\|\Pi_n f\|_{X^{s,b}_\alpha}\|_{\ell^{p}_n}.
	\end{equation}
	For $p=2$, the space $X^{s,b}_{p,\alpha}$ simply reduces to the usual Bourgain's space $X^{s,b}_\alpha$.
\end{definition}


\begin{definition}\label{X^{s,b}_p(I)}
Given a time interval I, we define the local-in-time spaces $X^{s,b}_p(I):=X^{s,b}_p(\mathbb{R}\times I)$ as the collection of functions $u$ such that
\begin{equation*}
	\|u\|_{X^{s,b}_{p,\alpha}(\mathbb{R}\times I)}:=\inf \left\{ \|v\|_{X^{s,b}_{p,\alpha}}: v|_{\mathbb{R}\times I}=u\right\}
\end{equation*}
is finite.

\end{definition}

For the sake of simplicity in exposition we will adopt the notations  $X^{s,b}_{2, \alpha}:=X^{s,b}_{ \alpha}$ and $X^{s,b}_{p, 1}:=X^{s,b}_{ p}$.

In the following lemma we record some properties of the  Bourgain type spaces $X^{s,b}_{p,\alpha}$ adapted to the modulation spaces $M^{2,p}_s(\mathbb{R})$.

\begin{lemma}\label{imersao dos espacos de Bourgain tipo modulacao}
	Let $\alpha, s\in\mathbb{R}$, $\alpha\neq 0$, $b>\frac{1}{2}$ and $p\geq 1$, then
	\begin{equation}\label{Imersion_M_space}
		X^{s,b}_{p,\alpha}\subset C(\mathbb{R}; M^{2,p}_s(\mathbb{R})).
	\end{equation} 
	For $p\geq q\geq 1$, we have
	\begin{equation}\label{xpxq}
		\|u\|_{X^{s,b}_{p,\alpha}}\leq
		\|u\|_{X^{s,b}_{q,\alpha}}.
	\end{equation}
For $p\geq q\geq 1$, it holds that
\begin{equation}\label{estimate X q and p}
	\|P_N u\|_{X^{s,b}_{q,\alpha}}\lesssim N^{\frac{1}{q}-\frac{1}{p}}\|P_N u\|_{X^{s,b}_{p,\alpha}}.
\end{equation}
For $2\leq p< \infty$ and any $\epsilon>0$, the following estimate holds
\begin{equation}\label{N-epsilon}
	\sum_{\begin{smallmatrix}
			N\geq 1\\
			\text{dyadic}
	\end{smallmatrix}} 
	N^{-\epsilon}
	\|P_N u\|_{X^{s,b}_{p,\alpha}}\lesssim  
	\|u\|_{X^{s,b}_{p,\alpha}}.
\end{equation}

\end{lemma}
\begin{proof}

The proof of \eqref{Imersion_M_space} is similar to the proof used in standard Bourgain's spaces. To prove \eqref{xpxq}, we can use the embedding $\ell^q(\mathbb{Z})\subset \ell^p(\mathbb{Z})$ for $1\leq q\leq p$.

To prove \eqref{estimate X q and p}, we use the definition of $X^{s,b}_{q,\alpha}$-norm. In fact, applying H\"older's inequality with $\frac{1}{q}=\frac{1}{p}+\frac{1}{r}$, we obtain
\begin{equation*}
	\begin{split}
		\|P_N u\|_{X^{s,b}_{q, \alpha} }
		& = \left (\sum_{n\in\mathbb{Z}} 
		\langle n\rangle^{sq} \|\langle \tau-\alpha\xi^3\rangle^b \widehat{P_N u}\|^q_{L^2_{\tau,\xi}(\mathbb{R},[n,n+1])}
		\right)^{\frac{1}{q}}\\
		& = \{\sum_{c_1N\leq n \leq c_2N} 
		{ (c_n.1)}^q \}^{\frac{1}{q}}\\
		& \leq \{\sum_{c_1N\leq n \leq c_2N} {c_n}^p\}^{\frac{1}{p}} 
		\{\sum_{c_1N\leq n \leq c_2N} 1^r\}
		^{\frac{1}{r}}\\
		& \sim N^{\frac{1}{r}} 
		\|P_N u\|_{X^{s,b}_{p, \alpha}}\\
		& = N^{\frac{1}{q}-\frac{1}{p}} \|P_N u\|_{X^{s,b}_{p, \alpha}},
	\end{split}
\end{equation*}
where $c_n=\langle n\rangle^{s} \|\langle \tau-\alpha\xi^3\rangle^b \widehat{P_N u}\|_{L^2_{\tau,\xi}(\mathbb{R},[n,n+1])}$.
\end{proof}

\begin{lemma}
Assume that $f,f_1,f_2$ and $f_3$ belong to Schwartz space in $\mathbb{R}^2$. Then we have
$$\int_{\mathbb{R}^6}\overline{\widehat{f}(\xi,\tau)}
\widehat{f_1}(\xi_1,\tau_1)\widehat{f_2}(\xi_2,\tau_2)\widehat{f_3}(\xi_3,\tau_3)
d\xi_1 d\tau_1 d\xi_2 d\tau_2 d\xi d\tau
=\int_{\mathbb{R}^2}\overline{f} f_1 f_2 f_3(x,t)dxdt,$$ 
where \,$\xi_3=\xi-\xi_1-\xi_2$ \,and\, $\tau_3=\tau-\tau_1-\tau_2$, i.e. $\xi_1+\xi_2+\xi_3=\xi$ \,and\, $\tau_1+\tau_2+\tau_3=\tau$.
\end{lemma}
\begin{proof}
This lemma easily follows using the Plancherel's identity and the convolution theorem for the Fourier transform (see Lemma 7 in \cite{Carvajal_3}).
\end{proof}

\begin{definition}
We say that $f=g+O(1)$\, if and only if \,$|f-g|\lesssim 1$.
\end{definition}

\begin{proposition}\label{O(1)}
{\bf  i)} Let $|x|\gg 1$, \,and\, $x=n+O(1)$, then
$$|x|\sim |n|\gg 1.$$
{\bf  ii)} Let $a,b\in\mathbb{R}$ such that $ab\neq 0$. If 
$
|a-b|\leq |a+b|
$
then one has $ab>0$ and $|a+b|=|a|+|b|$.
\\
{\bf  iii)} Let $a,b\in\mathbb{R}$, then
\begin{equation}\label{O(2)}
	\max\{|a-b|,|a+b|\}=|a|+|b|,
\end{equation}
and
\begin{equation*}
	\max\{|a|,|b|\}=\frac{|a+b|}{2}+\frac{|a-b|}{2}.
\end{equation*}
\end{proposition}
\begin{proof}
See propositions 3.7, 3.8 and 3.9 in \cite{Bar:03}.
\end{proof}

\begin{proposition}\label{sumatorio lp_lp'}
{\bf 1.} \,Let $a,b\in\mathbb{Z}$, and $c_1+c_2>1$. Then, we have
\begin{equation}\label{serie1}
	S(a,b):=\sum_{\begin{smallmatrix}
			n\in\mathbb{Z}\\n \neq a,b
	\end{smallmatrix}}\frac{1}{|n-a|^{c_1}|n-b|^{c_2}}
	< C < \infty,
\end{equation}
where C is independent of $a$ and $b$.
\\
{\bf 2.} \,Let $r\in \R$, $p>1$. Then, one has
\begin{equation}\label{sumcomreal}
\sum_{\begin{smallmatrix}
			n\in\mathbb{Z}\\|n- r|>1
	\end{smallmatrix}}
\frac{1}{|n-r|^{p}}
	< C < \infty,
\end{equation}
where C is independent of $r$.
\\
{\bf 3.} \,Let $\frac{1}{p}+\frac{1}{p'}=1$, $p\geq 1$. Then, for each $\epsilon>0$, the following estimate holds
\begin{equation}\label{sumatorio lp_lp' enunciado}
	\sum_{\begin{smallmatrix}
			m,n\in\mathbb{Z}\\m\neq n
	\end{smallmatrix}}
	\frac{a_m b_n}{|m-n|\langle n\rangle^\epsilon}
	\leq c_\epsilon \|a_m \|_{\ell^p(\mathbb{Z})}\|b_{n}
	\|_{\ell^{p'}(\mathbb{Z})}.
\end{equation} 
\\
{\bf 4.} \,Let $M\gg1$ and $\kappa>0$. Then
\begin{equation}\label{sumkM}
	\sum_{|n| \sim M}
	\frac{1}{|n|^{\kappa}}
	\lesssim\frac{1}{M^{\kappa-1}}.
\end{equation} 
\end{proposition}
\begin{proof}
	The inequality \eqref{serie1} follows from the convergence of the series $\sum_{n\in \Z^+}  {\frac{1}{n^p}}$ for $p>1$ and  \eqref{sumatorio lp_lp' enunciado} follows by applying Hölder’s and Young’s inequalities.
In order to prove \eqref{sumcomreal}, first note that if $r\in \Z$ the inequality is trivial. Therefore, we consider the case when $r\notin \Z$ and proceed as follows. Since $\R=\cup_{n\in \Z} [n, n+1)$,  there exist $n_r \in \Z$ such that $r\in [n_r, n_r+1)$. Thus,
\begin{equation*}
\begin{split}
\sum_{\begin{smallmatrix}
			n\in\mathbb{Z}\\|n- r|>1
	\end{smallmatrix}}
\frac{1}{|n-r|^{p}} &= \sum_{\begin{smallmatrix}
			n\in\mathbb{Z}\\ n-r>1\end{smallmatrix}} \frac1{(n-r)^p}+\sum_{\begin{smallmatrix}
			n\in\mathbb{Z}\\ r-n>1\end{smallmatrix}} \frac1{(r-n)^p}\\
&\leq \sum_{n> n_r+1} \frac1{(n-n_r-1)^p}+\sum_{n<n_r} \frac1{(n_r-n)^p}\\
&\lesssim \sum_{\begin{smallmatrix}
			n\in\mathbb{Z}\\|n|\geq1
	\end{smallmatrix}} \frac{1}{|n|^p}=C<\infty.
\end{split}
\end{equation*}
\end{proof}

\section{Local Well-posedness for the mKdV type System}\label{chapter 3}

This section is devoted to derive linear and multilinear estimates that play crucial role in the proof of the local well-posedness result. We start with the linear estimates.

\subsection{Linear estimates}
The following lemmas provide the homogeneous and non-homogeneous linear estimates whose  proofs are very similar to those commonly known in the classical Bourgain's spaces, see for example \cite{CKSTT}, \cite{Ginibre}, \cite{Tao} and references therein. So, we only state them without detailed proof.
\begin{lemma}\label{Lema 2.1}[Homogeneous Linear Estimates]
	Let $f\in M^{2,p}_s(\mathbb{R})$, $\eta\in C^\infty_c(\mathbb{R})$ be a cut-off function, and let  $\{S_\lambda(t)\}_{t\in\mathbb{R}}$ be the unitary
	group associated with the IVP (\ref{PVI_Majda-Biello equation}) given by $\widehat{S_\lambda(t)f}(\xi)=e^{-it \lambda \xi^3}\widehat{f}(\xi)$. For any $s,b\in\mathbb{R}$, $1\leq p< \infty$, $\lambda\in \R$  and $0< T\leq 1$, the following estimate hold
	\begin{equation}\label{Est_lineal_2}   
		\|\eta(t) S_\lambda(t)f\|_{X^{s,b}_{p,\alpha}([0,T])}
		\lesssim \|f\|_{M^{2,p}_s(\mathbb{R})}.
	\end{equation}
\end{lemma}

\begin{lemma}\label{Lema 2.2}[Non-Homogeneous Linear Estimates]
	Let $F\in X^{s,b'}_p(I)$, $\eta_T\in C^\infty_c(\mathbb{R})$ be a cut-off function, and let  $\{S_\lambda(t)\}_{t\in\mathbb{R}}$ be the unitary
	groups associated with the IVP (\ref{PVI_Majda-Biello equation}). For any $s,b,b'\in\mathbb{R}$ such that $-1/2<b'\leq 0\leq b\leq b'+1$, $1\leq p <\infty$, $\lambda\in \R$, and $I=[0,T]$ with $0<T \leq 1$, the following estimate hold
	\begin{equation}\label{estimativ_lineales_noh}
		\|\eta_T(t) \int_0^t S_{\lambda}(t-t')F(t')dt'\|_{X^{s,b}_p(I)}
		\lesssim T^{1-b+b'}\|F\|_{X^{s,b'}_p(I)}.
	\end{equation}
\end{lemma}


\subsection{Bilinear estimates }  In what follows, we record some bilinear estimates that will be useful throughout this work.
\begin{definition}
	For given $\theta>0$, let $I^\theta:=(-\partial^2_x)^{\frac{\theta}{2}}$ be the Riesz potential of order $-\theta$, and let $I^\theta_-$ be defined by
	\begin{equation*}
		\mathcal{F}_x(I^\theta_-(f,g))(\xi):=\int_{\xi=\xi_1+\xi_2}
		|\xi_1-\xi_2|^\theta \widehat{f}(\xi_1) \widehat{g}(\xi_2)\,d\xi_1.
	\end{equation*}
\end{definition}

\begin{lemma}\label{Op de Riesz}
	Let $I^{\frac{1}{2}}$ and $I^\frac{1}{2}_-$ be the operators as defined above (with $\theta=\frac{1}{2}$). Then the following estimate holds
	\begin{equation*}
		\|I^{\frac{1}{2}} I^\frac{1}{2}_-(u_1,u_2)\|_{L^2_{x,t}}
		\lesssim \|u_1\|_{X^{0,\frac{1}{2}+}} \|u_2\|_{X^{0,\frac{1}{2}+}}.
	\end{equation*}
\end{lemma}
\begin{proof}
	See \cite{Bar:04}, see also Lemma 3.1 and Corollary 3.2 in \cite{Grunrock_2}.
\end{proof}

\begin{corollary}\label{CorolPPi1}
	Let $N_1,N_2\geq 1$ be dyadic integers such that $N_1\gg N_2$.  Then, we have
	\begin{equation}\label{corolario_1}
		\|P_{N_1}v P_{N_2}u\|_{L^2_{x,t}(\mathbb{R}^2)}
		\lesssim \frac{1}{N_1}\|P_{N_1}v\|_{X^{0,\frac{1}{2}+}_{\alpha}} 
		\|P_{N_2}u\|_{X^{0,\frac{1}{2}+}_{\alpha}}.
	\end{equation}

	Let $m,n\in\mathbb{Z}$ such that $|m+n|,|m-n|>2$. Then, we have
	\begin{equation*}
		\|\Pi_m v \Pi_n u\|_{L^2_{x,t}(\mathbb{R}^2)}\lesssim 
		\frac{1}{\sqrt{|m+n||m-n|}}
		\|\Pi_m v\|_{X^{0,\frac{1}{2}+}} 
		\|\Pi_n u\|_{X^{0,\frac{1}{2}+}}.
	\end{equation*}
\end{corollary}
\begin{proof}
See Corollary 1 in \cite{Bar:02}.
\end{proof}


\begin{lemma}\label{lemma alpha difert to 0}
	Let $u,v\in \mathbb{S}(\mathbb{R}^2)$ such that $supp\,\widehat{u}\subset \{|\xi|\sim N\}$, $supp\,\widehat{v}\subset \{|\xi|\ll N\}$ and $\alpha\neq 0$. Then
	\begin{equation}\label{lema_esti_bilineal_con alfa}
		\|u v\|_{L^2_x L^2_t}\lesssim \frac{1}{N}\|u\|_{X^{0,\frac{1}{2}+}_\alpha}
		\|v\|_{X^{0,\frac{1}{2}+}}.
	\end{equation}
\end{lemma}
\begin{proof}
	Using arguments from \cite{Carvajal_2} and \cite{Colliander_1}, to prove the estimate (\ref{lema_esti_bilineal_con alfa}), it suffices to show that
	\begin{equation*}
		\|u v\|_{L^2_x L^2_t}\lesssim \frac{1}{N} \|\phi\|_{L^2} \|\psi\|_{L^2},
	\end{equation*}
	where $u=S_\alpha(t)\phi$ and $v=S(t)\psi$.\\
	Using duality and Fubini's theorem, we have
	\begin{equation}\label{eqlemx1}
		\begin{split}
			\|u v\|_{L^2_x L^2_t}
			& = \sup_{\|F\|\leq 1}\left|\int_{\mathbb{R}^2} uv \overline{F}\, dxdt\right|\\ 
			& = \sup_{\|F\|\leq 1}\left|\int_{\mathbb{R}}\int_{\xi_1+\xi_2=\xi_3}\widehat{u}(\xi_1) 
			\widehat{v}(\xi_2)\overline{\widehat{F}}(\xi_3)\, d\xi_1 d\xi_2dt\right|\\
			& = \sup_{\|F\|\leq 1}\left|\int_{\mathbb{R}}\int_{\xi_1+\xi_2=\xi_3}e^{it\alpha \xi_1^3}\widehat{\phi}(\xi_1) 
			e^{it \xi_2^3}\widehat{\psi}(\xi_2)\overline{\widehat{F}}(\xi_3)\, d\xi_1 d\xi_2dt\right|\\
			& = \sup_{\|F\|\leq 1}\left|\int_{\xi_1+\xi_2=\xi_3}
			\widehat{\phi}(\xi_1) \widehat{\psi}(\xi_2)
			\int_{\mathbb{R}}
			e^{it\alpha \xi_1^3+it\xi_2^3}
			\overline{\widehat{F}}(\xi_3)\, dt d\xi_1 d\xi_2\right|,
		\end{split}
	\end{equation}
	where we used the notation $\|F\|:=\|F\|_{L^2_x L^2_t}$. Thus, we get
	\begin{equation}\label{eqlemx2}
		\begin{split}
			\|u v\|_{L^2_x L^2_t}            
			& = \sup_{\|F\|\leq 1}\left|\int_{\xi_1+\xi_2=\xi_3}
			\widehat{\phi}(\xi_1) \widehat{\psi}(\xi_2)
			\overline{\int_{\mathbb{R}}
				e^{-it(\alpha\xi_1^3+\xi_2^3)}
				\widehat{F}(\xi_3,t)\, dt} d\xi_1 d\xi_2\right|\\
			& = \sup_{\|F\|\leq 1}\left|\int_{\xi_1+\xi_2=\xi_3}
			\widehat{\phi}(\xi_1) \widehat{\psi}(\xi_2)
			\overline{\widehat{F}(\xi_3,\alpha\xi_1^3+\xi_2^3)}\, d\xi_1 d\xi_2\right|\\
			& \leq \sup_{\|F\|\leq 1}\int_{\xi_1+\xi_2=\xi_3}
			\left|\widehat{\phi}(\xi_1) \widehat{\psi}(\xi_2)
			\overline{\widehat{F}(\xi_3,\alpha\xi_1^3+\xi_2^3)}\right|\, d\xi_1 d\xi_2\\
			& = \sup_{\|F\|\leq 1}\int_{\mathbb{R}^2}
			\left|\widehat{\phi}(\xi_1) \widehat{\psi}(\xi_2)
			\overline{\widehat{F}(\xi_1+\xi_2,\alpha\xi_1^3+\xi_2^3)}\right|\, d\xi_1 d\xi_2\\
			& \sim \sup_{\|F\|\leq 1}\int_{\mathbb{R}^2}
			\left|\widehat{\phi}(\xi_1(s,r)) 
			\widehat{\psi}(\xi_2(s,r))
			\overline{\widehat{F}(s,r)}\right|\,\frac{1}{N^2} ds dr,
		\end{split}
	\end{equation}
	in the last inequality, we made a change of variables, setting $\xi_1+\xi_2=s$ and $\alpha \xi_1^3+\xi_2^3=r$, with the corresponding Jacobian 
	$$
	\left|\dfrac{\partial(s,r)}{\partial(\xi_1, \xi_2)}\right|=3|(\xi_2^2- \alpha\xi_1^2)| \sim N^2.
	$$
	 Now, due to H\"older's inequality,
	\begin{equation*}
		\begin{split}
			\|u v\|_{L^2_x L^2_t} & \leq \sup_{\|F\|\leq 1}
			\frac{1}{N^2}
			\|\widehat{\phi}\widehat{\psi}\|_{L^2_{s,r}} \|\widehat{F}\|_{L^2_{s,r}}\\
			& = \sup_{\|F\|\leq 1} \frac{1}{N^2}
			\left(\int_{\mathbb{R}^2} |\widehat{\phi}(\xi_1(s,r))|^2 |\widehat{\psi}(\xi_2(s,r))|^2ds dr\right)^\frac{1}{2} \|F\|_{L^2_x L^2_t}\\
			& \sim \sup_{\|F\|\leq 1} \frac{1}{N^2}
			\left(\int_{\mathbb{R}^2} |\widehat{\phi}(\xi_1)|^2|\widehat{\psi}(\xi_2)|^2 N^2 d\xi_1 d\xi_2\right)^\frac{1}{2} \|F\|_{L^2_x L^2_t}\\
			& = \frac{1}{N} \|\phi\|_{L^2} \|\psi\|_{L^2},
		\end{split}
	\end{equation*}
	in the third line, we performed the inverse change of variables, returning to the original variables $\xi_1$, $\xi_2$. This completes the proof of the  estimate (\ref{lema_esti_bilineal_con alfa}).
\end{proof}

\begin{corollary}\label{Pi_mnalpha}
	Let $m,n\in\mathbb{Z}$ be such that $|m\pm \sqrt{\alpha}\, n|\gg1 $. Then, we have
	\begin{equation*}
		\|\Pi_m v \Pi_n u\|_{L^2_{x,t}(\mathbb{R}^2)}\lesssim 
		\frac{1}{|m+\sqrt{\alpha}\, n|^{1/2}|m-\sqrt{\alpha}\, n|^{1/2}}
		\|\Pi_m v\|_{X^{0,\frac{1}{2}+}_{\alpha}} 
		\|\Pi_n u\|_{X^{0,\frac{1}{2}+}}.
	\end{equation*}
\end{corollary}

\begin{proof}
We use similar argument as in \eqref{eqlemx1} and \eqref{eqlemx2}. In this case, the Jacobian is
$$
	\left|\dfrac{\partial(s,r)}{\partial(\xi_1, \xi_2)}\right|=3|(\xi_2^2- \alpha\xi_1^2)|=3|\xi_2- \sqrt{\alpha}\,\xi_1|\,|\xi_2+\sqrt{\alpha}\,\xi_1|,
$$
where $\xi_1=n+O(1)$ and $\xi_2=m+O(1)$.  Observe that
$$
|\xi_2\pm \sqrt{\alpha}\,\xi_1|=|m \pm \sqrt{\alpha}\,n+ O(1)| \sim |m \pm \sqrt{\alpha}\,n|.
$$
\end{proof}


\subsection{Trilinear estimates}
In this subsection we derive  trilinear estimates that will be fundamental in establishing the local well- posedness result for the IVP \eqref{PVI_Majda-Biello equation} with initial data in the modulation space $M^{2,p}_s$ where $p\ge2$ and $s> \frac{1}{4}-\frac1{p}$.

\begin{proposition}\label{estimativ bilinear u1u2}
	Let $\alpha\in (0, 1)$,  $2\leq p <\infty$, and   $0<T\leq 1$. Given $v\in X^{s,b}_p([0,T])$ and $w_1,w_2\in X^{s,b}_{p,\alpha}([0,T])$, the following trilinear estimate
	\begin{equation}\label{trilinear estimate_1}
		\|\partial_x (v\, w_1 w_2)
		\|_{X^{s,b'}_p([0,T])}
		\lesssim
		\|v\|_{X^{s,b}_p([0,T])}
		\|w_1\|_{X^{s,b}_{p,\alpha}([0,T])} 
		\|w_2\|_{X^{s,b}_{p,\alpha}([0,T])},
	\end{equation}
	hold for each $s> \frac{1}{4}-\frac1{p}$ and for some $b>\frac{1}{2}$, $b'>-\frac{1}{2}$. 
	
	 Similarly, given $v_1, v_2\in X^{s,b}_p([0,T])$ and $w\in X^{s,b}_{p,\alpha}([0,T])$, the following trilinear estimate
	\begin{equation}\label{trilinear estimate_2}
		\|\partial_x (v_1 v_2\, w)
		\|_{X^{s,b'}_{p,\alpha}([0,T])}
		\lesssim
		\|v_1\|_{X^{s,b}_p([0,T])}
		\|v_2\|_{X^{s,b}_p([0,T])} 
		\|w\|_{X^{s,b}_{p,\alpha}([0,T])},
	\end{equation}
	hold for each $s> \frac{1}{4}-\frac1{p}$, and for some $b>\frac{1}{2}$, $b'>-\frac{1}{2}$.
\end{proposition}

\begin{proof}
	We will provide a detailed proof for the estimate \eqref{trilinear estimate_1}. The proof of the estimate \eqref{trilinear estimate_2} follows with an analogous argument.
	
	Given  $0<\epsilon\ll 1$, let us define $b'=-\frac{1}{2}+2\epsilon$ and $b=\frac{1}{2}+\epsilon$. In this setting, to get \eqref{trilinear estimate_1} we need to prove that
	\begin{equation}\label{Estimativ_Bilinearl_Prop}
		\|\partial_x (v w_1 w_2)\|_{X^{s,-\frac{1}{2}+2\epsilon}_p}\lesssim
		\|v\|_{X^{s,\frac{1}{2}+\epsilon}_p}
		\|w_1\|_{X^{s,\frac{1}{2}+\epsilon}_{p,\alpha}} 
		\|w_2\|_{X^{s,\frac{1}{2}+\epsilon}_{p,\alpha}},
	\end{equation}
	for $s> \frac{1}{4}-\frac1{p}$.
	
	Using a duality argument, one has
	\begin{equation}\label{dual-I}
		\|\partial_x(v w_1 w_2)\|_{X^{s,-\frac{1}{2}+2\epsilon}_p}=\sup_{
			\begin{smallmatrix}
				\widetilde{z} \,\in\, X^{-s,\frac{1}{2}-2\epsilon}_{p'}\\
				\|\widetilde{z}\|\leq 1	\end{smallmatrix}}\left|\int_{\mathbb{R}^2}\partial_x(v w_1 w_2)\widetilde{z}\,dx dt\right|,
	\end{equation}
	where $\frac{1}{p}+\frac{1}{p'}=1$ and $\|\widetilde{z}\|=\|\widetilde{z}\|_{X^{-s,\frac{1}{2}-2s}_{p'}}$.\\
	In the light of \eqref{dual-I}, proving the trilinear estimate (\ref{Estimativ_Bilinearl_Prop}) is equivalent to proving 
	\begin{equation}\label{Estimativ_Bilinearl_Prop_2}
		\left|\int_{\mathbb{R}^2}  v w_1 w_2 \langle\partial_x\rangle^s \partial_x z\, dxdt\right|\lesssim 
		\|v\|_{X^{s,\frac{1}{2}+\epsilon}_p}
		\|w_1\|_{X^{s,\frac{1}{2}+\epsilon}_{p,\alpha}} \|w_2\|_{X^{s,\frac{1}{2}+\epsilon}_{p,\alpha}} \|z\|_{X^{0,\frac{1}{2}-2\epsilon}_{p'}},
	\end{equation}
	with $\widehat{\widetilde{z}}=\langle\xi\rangle^s \widehat{z}$ e $\widehat{\langle\partial_x\rangle^s z}(\xi)=\langle\xi\rangle^s\widehat{z}$.
	
	Using Fourier transform properties, proving (\ref{Estimativ_Bilinearl_Prop_2})  is equivalent to proving
	\begin{equation}\label{Estimativ_Bilinearl_Prop_3}
		\left|\int_{\begin{smallmatrix}
				\xi+\xi_1+\xi_2+\xi_3=0\\
				\tau+\tau_1+\tau_2+\xi_3=0
		\end{smallmatrix}}\xi 
		\langle\xi\rangle^s \widehat{v} \widehat{w_1}\widehat{w_2}       
		\widehat{z}\right|
		\lesssim 
		\|v\|_{X^{s,\frac{1}{2}+\epsilon}_p}
		\|w_1\|_{X^{s,\frac{1}{2}+\epsilon}_{p,\alpha}} \|w_2\|_{X^{s,\frac{1}{2}+\epsilon}_{p,\alpha}}   
		\|z\|_{X^{0,\frac{1}{2}-2\epsilon}_{p'}}.
	\end{equation}
	
	In what follows, we use $\xi_{max},\, \xi_{med},\, \xi_{min}$ to denote a rearrangement of $\xi_1,\, \xi_2,\, \xi_3$ such that $|\xi_{max}|\geq |\xi_{med}|\geq |\xi_{min}|$.
	Since $\xi+\xi_1+\xi_2+\xi_3=0$, we have $|\xi|=|\xi_1+\xi_2+\xi_3|\lesssim |\xi_{max}|$.
	
	Also, we will use a dyadic decomposition $|\xi_j|\sim N_j$ and $|\xi|\sim N$ for $N_j, N\geq 1$ dyadic. In this case, we use the notation $N_{max}\sim |\xi_{max}|$, $N_{med}\sim |\xi_{med}|$, and $N_{min}\sim |\xi_{min}|$.
	
	We use $\sigma,\sigma_j$ to denote
	$$\sigma:=\tau-\xi^3, \quad \sigma_j:=\tau_j-\alpha \xi^3_j, \,\, j=1,2,\quad \textrm{and} \quad \sigma_3:=\tau_3- \xi^3$$
	and 
	\begin{equation*}
		\sigma_{max}:=\max\{|\sigma|,|\sigma_1|,|\sigma_2|,|\sigma_3|\}.
	\end{equation*}
	Finally, we use the notation 
	\begin{equation*}
		u_N:=P_N u \;\text{ and }\; u_n:=\Pi_n u,
	\end{equation*}
	where $P_N$ is the Littlewood-Paley projector and $\Pi_n$ is as defined in (\ref{Projetor L-P}).
	
	With these notations, using Parseval's identity and inverse Fourier transform, proving (\ref{Estimativ_Bilinearl_Prop_3}) is equivalent
	to proving
	\begin{equation}\label{Estimativ_Bilinearl_Prop_4}
		\begin{split}        
			I&:=\sum_{\begin{smallmatrix}
					N_1,N_2,N_3,N\geq 1\\
					\text{dyadic}
			\end{smallmatrix}}
			N^{s+1}\left|
			\int_{\mathbb{R}^2} 
			(w_1)_{N_1}
			(w_2)_{N_2}
			v_{N_3}
			z_N \,dx dt\right|\\
			& \lesssim 
			\|w_1\|_{X^{s,\frac{1}{2}+}_{p,\alpha}} \|w_2\|_{X^{s,\frac{1}{2}+}_{p,\alpha}} 
			\|v\|_{X^{s,\frac{1}{2}+}_p}
			\|z\|_{X^{0,\frac{1}{2}-2\epsilon}_{p'}}.
		\end{split}
	\end{equation}
	
	We prove the estimate \eqref{Estimativ_Bilinearl_Prop_4} dividing in the following four different cases
	\\
	{\bf A. Trivial cases:} $|\xi_{max}|\lesssim 1$ or $\langle \sigma_{max}\rangle\gg N_{max}^{10}$,\\
	{\bf B. Non-resonant case:} $N_{max}\gg N_{med}\geq N_{min}$,\\
	{\bf C. Semi-resonant case:} $N_{max}\sim N_{med}\gg N_{min}$,\\
	and\\
	{\bf D. Resonant case:} $N_{max}\sim N_{med}\sim N_{min}$.\\

	\noindent
	\textbf{Case A. Trivial cases:} In this case, we prove \eqref{Estimativ_Bilinearl_Prop_4} further dividing in  two subcases $|\xi_{max}|\lesssim 1$ or $\langle \sigma_{max}\rangle\gg \langle\xi_{max}\rangle^{10}$.\\
	
	\noindent
	\textbf{Subcase A1.} $|\xi_{max}|\lesssim 1$: In this case $|\xi|=|\xi_1+\xi_2+\xi_3|\leq 3|\xi_{max}|\lesssim 1$ and $N\lesssim N_{max}$. Thus,
	using the fact $s+1=\frac{5}{4}-\frac1p>0$,  we apply H\"older's inequality, followed by Bernstein's inequality, to get
	\begin{equation*}
		\begin{split}
			I 
			& \lesssim \sum_{N_{max},N\lesssim 1} N^{s+1}_{max}
			\left|\int_{\mathbb{R}^2}
			(w_1)_{N_1}
			(w_2)_{N_2} v_{N_3} z_N dxdt\right|\\
			&  \lesssim \sum_{N_{max},N\lesssim 1} N^{s+1}_{max}
			\|(w_1)_{N_1}\|_{L^2_{x,t}}
			\|(w_2)_{N_2}\|_{L^\infty_{x,t}}
			\|v_{N_3}\|_{L^\infty_{x,t}}
			\|z_N\|_{L^2_{x,t}}\\
			& \lesssim \sum_{N_{max},N\lesssim 1} N^{s+2}_{max}
			\|(w_1)_{N_1}\|_{L^2_{x,t}}
			\|(w_2)_{N_2}\|_{L^\infty_t L^2_x}
			\|v_{N_3}\|_{L^\infty_t L^2_x}
			\|z_N\|_{L^2_{x,t}}.
		\end{split}
	\end{equation*}
	Next, we will apply the Sobolev embedding theorem, the definitions of the norms $X^{s,b}_\alpha(\mathbb{R}^2)$ and $X^{s,b}(\mathbb{R}^2)$, and their respective properties, to obtain
	\begin{equation*}
		\begin{split}
			I             
			& \lesssim \sum_{N_{max},N\lesssim 1} N^{s+2}_{max}
			\|(w_1)_{N_1}\|_{X^{0,\frac{1}{2}+\epsilon}_\alpha}
			\|(w_2)_{N_2}\|_{X^{0,\frac{1}{2}+\epsilon}_\alpha}
			\|v_{N_3}\|_{X^{0,\frac{1}{2}+\epsilon}}       
			\|z_N\|_{X^{0,\frac{1}{2}-2\epsilon}}\\
			& \sim \sum_{N_{max},N\lesssim 1} N^{s+2}_{max}
			N_1^{-s}\|(w_1)_{N_1}\|_{X^{s,\frac{1}{2}+\epsilon}_\alpha}
			N_2^{-s}\|(w_2)_{N_2}\|_{X^{s,\frac{1}{2}+\epsilon}_\alpha}
			N_3^{-s}
			\|v_{N_3}\|_{X^{0,\frac{1}{2}+\epsilon}}
			\|z_N\|_{X^{0,\frac{1}{2}-2\epsilon}}\\
			& \lesssim \sum_{N_{max},N\lesssim 1} N^{-2s+2}_{max}
			\|(w_1)_{N_1}\|_{X^{s,\frac{1}{2}+\epsilon}_\alpha}
			\|(w_2)_{N_2}\|_{X^{s,\frac{1}{2}+\epsilon}_\alpha}
			\|v_{N_3}\|_{X^{0,\frac{1}{2}+\epsilon}}
			\|z_N\|_{X^{0,\frac{1}{2}-2\epsilon}}.
		\end{split}
	\end{equation*}
	Finally, using the Lemma \ref{imersao dos espacos de Bourgain tipo modulacao}, one gets
	\begin{equation}\label{casos triviales_1}
		\begin{split}
			I                 
			& \lesssim \sum_{N_{max},N\lesssim 1} N^{-2s+2}_{max}
			N_1^{\frac{1}{2}-\frac{1}{p}}\|(w_1)_{N_1}\|_{X^{s,\frac{1}{2}+\epsilon}_{p,\alpha}}
			N_2^{\frac{1}{2}-\frac{1}{p}}\|(w_2)_{N_2}\|_{X^{s,\frac{1}{2}+\epsilon}_{p,\alpha}}
			N_3^{\frac{1}{2}-\frac{1}{p}}
			\|v_{N_3}\|_{X^{0,\frac{1}{2}-2\epsilon}_p}
			\|z_N\|_{X^{0,\frac{1}{2}-2\epsilon}_{p'}},
		\end{split}
	\end{equation}
	 where in this case $N_{\max}^a \sim 1$ for $a \in \R$ is used. Summing over the dyadic blocks $ 1\leq N_1,N_2,N_3,N\lesssim 1$, estimate (\ref{casos triviales_1}) yields the required estimate (\ref{Estimativ_Bilinearl_Prop_4}).
	\\
	
	\noindent
	\textbf{Subcase A2.} $\langle \sigma_{max}\rangle\gg \langle\xi_{max}\rangle^{10}$: First, we consider the case $\langle \sigma_{max}\rangle=\langle \sigma_2\rangle$.  As in the previous
	case, we apply H\"older's inequality and Bernstein's inequality, to obtain
	\begin{equation*}
		\begin{split}
			I  \lesssim \sum_{N_1,N_2,N_3,N\geq 1} N^{s+2}_{max} \|(w_1)_{N_1}\|_{L^\infty_t L^2_x}
			\|(w_2)_{N_2}\|_{L^2_{x,t}}
			\|v_{N_3}\|_{L^\infty_t L^2_x}
			\|z_N\|_{L^2_{x,t}}.
		\end{split}
	\end{equation*}
	In the following, using the definitions of $X^{s,b}_\alpha(\mathbb{R}^2)$ and  $X^{s,b}(\mathbb{R}^2)$ norms and their properties, we get
	\begin{equation*}
		\begin{split}
			I 
			& \lesssim \sum_{N_1,N_2,N_3,N\geq 1} N^{s+2}_{max} 
			\|(w_1)_{N_1}\|_{X^{0,\frac{1}{2}+\epsilon}_\alpha}
			\|\langle\sigma_2\rangle^{-\frac{1}{2}-\epsilon}(w_2)_{N_2}\|_{X^{0,\frac{1}{2}+\epsilon}_\alpha}
			\|v_{N_3}\|_{X^{0,\frac{1}{2}+\epsilon}}
			\|z_N\|_{X^{0,\frac{1}{2}-2\epsilon}}\\     
			& \lesssim \sum_{N_1,N_2,N_3,N\geq 1} 
			N^{s+2}_{max} 
			\frac{1}{N_{max}^{5+10\epsilon}}
			N_1^{-s}\|(w_1)_{N_1}\|_{X^{s,\frac{1}{2}+\epsilon}_\alpha}
			N_2^{-s}\|(w_2)_{N_2}\|_{X^{s,\frac{1}{2}+\epsilon}_\alpha}
			N_3^{-s}\|v_{N_3}\|_{X^{0,\frac{1}{2}+\epsilon}}\\
			& \hspace{0.5cm} \times
			\|z_N\|_{X^{0,\frac{1}{2}-2\epsilon}}.
		\end{split}
	\end{equation*}
	Using the Lemma \ref{imersao dos espacos de Bourgain tipo modulacao}, one obtains
	\begin{equation}\label{caso trivial_2}
		\begin{split}
			I           
			& \lesssim \sum_{N_1,N_2,N_3,N\geq 1} 
			N^{-2s-3-10\epsilon+\frac{3}{2}-\frac{3}{p}}_{max}            
			\|(w_1)_{N_1}\|_{X^{s,\frac{1}{2}+\epsilon}_{p,\alpha}}
			\|(w_2)_{N_2}\|_{X^{s,\frac{1}{2}+\epsilon}_{p,\alpha}}
			\|v_{N_3}\|_{X^{0,\frac{1}{2}+\epsilon}_p}\\
			& \hspace{0.5cm}\times
			\|z_N\|_{X^{0,\frac{1}{2}-2\epsilon}_{p'}}
		\end{split}
	\end{equation}
	Since that $-2s-\frac32-\frac3{p}-10\epsilon<0$, summing over the dyadic blocks $N_1,N_2,N_3,N\geq 1$, the estimate (\ref{caso trivial_2}) yields the required estimate (\ref{Estimativ_Bilinearl_Prop_4}).
	
	The complementary cases $\langle\sigma_{max}\rangle=\langle\sigma_1\rangle$ or $\langle\sigma_{max}\rangle=\langle\sigma_3\rangle$ can be handled in an analogous manner.

	In what follows, we consider the case where
	\begin{equation}\label{casos_restantes}
		|\xi_{max}|\gg 1 \,\,\text{ and }\,\, \langle\sigma_{max}\rangle\lesssim
		\langle\xi_{max}\rangle^{10}.
	\end{equation}
	We split this into three diferent cases: the non-resonant case $N_{max}\gg N_{med}\geq N_{min}$,  the semi-resonant case $N_{max}\sim N_{med}\gg N_{min}$ and the resonant case $N_{max}\sim N_{med}\sim N_{min}$.	
\\

\noindent
	\textbf{Case B. Non-resonant case:} $N_{max}\gg N_{med}\geq N_{min}$\textbf{.}
	Using the symmetry of the nonlinearity, we can assume $N_1\gg N_2\geq N_3$. We observe that $N\lesssim 3N_{max}=3N_1$ and 
	$$N\sim |\xi|=|\xi_1+\xi_2+\xi_3|\geq |\xi_1|-|\xi_2|-\xi_3| \geq N_1 /2,$$
	so that $N\sim N_1$.
	
	Applying H\"older's inequality, Corollary \ref{CorolPPi1}, and condition (\ref{casos_restantes}), we obtain
	\begin{equation*}
		\begin{split}
			I 
			& \lesssim \sum_{N_1\sim N\gg N_2\geq N_3} N^{s+1}
			\frac{1}{N_1}
			\|(w_1)_{N_1}\|_{X^{0,\frac{1}{2}+\epsilon}_\alpha}
			\|(w_2)_{N_2}\|_{X^{0,\frac{1}{2}+\epsilon}_\alpha}
			\frac{1}{N}
			\|v_{N_3}\|_{X^{0,\frac{1}{2}+\epsilon}}
			\|z_N\|_{X^{0,\frac{1}{2}+\epsilon}}\\
			& \lesssim \sum_{N_1\sim N\gg N_2\geq N_3} 
			N^{s-1}
			\|(w_1)_{N_1}\|_{X^{0,\frac{1}{2}+\epsilon}_\alpha}
			\|(w_2)_{N_2}\|_{X^{0,\frac{1}{2}+\epsilon}_\alpha}
			\|v_{N_3}\|_{X^{0,\frac{1}{2}+\epsilon}}
			N^{30\epsilon}\|z_N\|_{X^{0,\frac{1}{2}-2\epsilon}}\\
			& \sim \sum_{N_1\sim N\gg N_2\geq N_3} 
			N^{s-1+30\epsilon-s}
			\|(w_1)_{N_1}\|_{X^{s,\frac{1}{2}+\epsilon}_\alpha}
			N_2^{-s}
			\|(w_2)_{N_2}\|_{X^{s,\frac{1}{2}+\epsilon}_\alpha}
			N_3^{-s}
			\|v_{N_3}\|_{X^{s,\frac{1}{2}+\epsilon}}      
			\|z_N\|_{X^{0,\frac{1}{2}-2\epsilon}}.        
		\end{split}
	\end{equation*}
	Using Lemma \ref{imersao dos espacos de Bourgain tipo modulacao}, we obtain
	\begin{equation}\label{non resonant case}
		\begin{split}
			I             
			& \lesssim 
			\sum_{N_1\sim N\gg N_2\geq N_3} 
			N^{-\frac{1}{2}+30\epsilon-\frac{1}{p}}   (N_2N_3)^{-s+\frac12-\frac{1}{p}} 
			\|(w_1)_{N_1}\|_{X^{s,\frac{1}{2}+\epsilon}_{p,\alpha}}
			\|(w_2)_{N_2}\|_{X^{s,\frac{1}{2}+\epsilon}_{p,\alpha}}
			\|v_{N_3}\|_{X^{s,\frac{1}{2}+\epsilon}_p}     
			\|z_N\|_{X^{0,\frac{1}{2}-2\epsilon}_{p'}}
		\end{split}
	\end{equation}
	Observe that if $-s+\frac12-\frac{1}{p}\leq 0$, then 
	$N^{-\frac{1}{2}+30\epsilon-\frac{1}{p}}   (N_2N_3)^{-s+\frac12-\frac{1}{p}} \lesssim N^{-\frac{1}{2}+30\epsilon-\frac{1}{p}}$
	and if $-s+\frac12-\frac{1}{p}> 0$, then 
	$
	N^{-\frac{1}{2}+30\epsilon-\frac{1}{p}} (N_2N_3)^{-s+\frac12-\frac{1}{p}} \leq N^{\frac12-2s -\frac{3}{p}+30\epsilon}
	$, where $\frac12-2s -\frac{3}{p}+30\epsilon<0$, since $s>\frac14-\frac3{2p}$.

	Applying in both cases \eqref{N-epsilon} with $\epsilon$ small, we get
	\eqref{Estimativ_Bilinearl_Prop_4}.
	\\
	
\noindent
	\textbf{Case C. Semi-resonant case:} $N_{max}\sim N_{med}\gg N_{min}$\textbf{.} Without loss of generality, it is enough to consider $N_{max}=N_1$ and $N_{med}=N_2$. We further divide this case into  two subcases, depending on the relationship between $N_{max}$ and $N$.\\
	
	\noindent
	\textbf{Subcase C1.} $N\ll N_{max}$\textbf{:} In this subcase, we have $N\ll N_2$ and $N_3\ll N_1$. Therefore, using H\"older's inequality and Corollary \ref{CorolPPi1}, we obtain
	\begin{equation*}
		\begin{split}
			I 
			& \lesssim \sum_{N_1\sim N_2\gg N_3,N}
			N^{s+1}
			\|(w_1)_{N_1} v_{N_3}\|_{L^2_{x,t}}
			\|(w_2)_{N_2} z_N\|_{L^2_{x,t}}\\
			& \lesssim \sum_{N_1\sim N_2\gg N_3,N}
			N^{s+1+30\epsilon}_{max}
			N_{max}^{-2}
			\|(w_1)_{N_1}\|_{X^{0,\frac{1}{2}+\epsilon}_\alpha}
			\|v_{N_3}\|_{X^{0,\frac{1}{2}+\epsilon}}
			\|(w_2)_{N_2}\|_{X^{0,\frac{1}{2}+\epsilon}_\alpha}
			\|z_N\|_{X^{0,\frac{1}{2}-2\epsilon}}\\
			& \sim \sum_{N_1\sim N_2\gg N_3,N}
			N^{-s-1+30\epsilon}_{max}
			\|(w_1)_{N_1}\|_{X^{s,\frac{1}{2}+\epsilon}_\alpha}
			N_3^{-s}\|v_{N_3}\|_{X^{s,\frac{1}{2}+\epsilon}}
			\|(w_2)_{N_2}\|_{X^{s,\frac{1}{2}+\epsilon}_\alpha}
			\|z_N\|_{X^{0,\frac{1}{2}-2\epsilon}}.
		\end{split}
	\end{equation*}
	Using Lemma \ref{imersao dos espacos de Bourgain tipo modulacao} and $s\leq \frac{1}{2}-\frac{1}{p}$ for $p\geq 2$, we obtain
	\begin{equation}\label{semi-resonant case}
		\begin{split}
			I
			& \lesssim \sum_{N_1\sim N_2\gg N_3,N}
			N^{-s-\frac{2}{p}+30\epsilon}_{max}
			N_{3}^{\frac{1}{2}-\frac{1}{p}-s}
			\|(w_1)_{N_1}\|_{X^{s,\frac{1}{2}+\epsilon}_{p,\alpha}}
	\|v_{N_3}\|_{X^{s,\frac{1}{2}+\epsilon}_p}
			\|(w_2)_{N_2}\|_{X^{s,\frac{1}{2}+\epsilon}_{p,\alpha}}
			\|z_N\|_{X^{0,\frac{1}{2}-2\epsilon}_{p'}}\\
			& \leq \sum_{N_1\sim N_2\gg N_3,N}
			N^{-2s-\frac{3}{p}+\frac12+30\epsilon}_{max}            
			\|(w_1)_{N_1}\|_{X^{s,\frac{1}{2}+\epsilon}_{p,\alpha}}
			\|v_{N_3}\|_{X^{s,\frac{1}{2}+\epsilon}}      
			\|(w_2)_{N_2}\|_{X^{s,\frac{1}{2}+\epsilon}_{p,\alpha}}
			\|z_N\|_{X^{0,\frac{1}{2}-2\epsilon}_{p'}}
		\end{split}
	\end{equation}
	where $-2s-\frac{3}{p}+\frac12+30\epsilon<0$ since $s>\frac14- \frac3{2p}$ and taking $0<30\epsilon < 2s+\frac{3}{p}-\frac12$. Applying \eqref{N-epsilon} in the estimate (\ref{semi-resonant case}), we obtain \eqref{Estimativ_Bilinearl_Prop_4}.
	\\
	
	\noindent
	\textbf{Subcase C2.} $N\sim N_{max}$\textbf{:} In this case, the condition $\xi_1+\xi_2+\xi_3+\xi=0$, implies that $|\xi_1+\xi_2+\xi|=|\xi_3|\sim N_3\ll N\sim N_{max}$. Consequently, not all $\xi_1,\, \xi_2,\, \xi$ can have the same sign. Therefore, we must have $\xi_1 \xi_2 < 0$ or $\xi_1 \xi < 0$ or $\xi_2 \xi < 0$.\\
	We will consider the case $\xi_1 \xi_2 < 0$, the other cases are similar.
	
	First note that
	\begin{enumerate}
		\item [$\bullet$] $|\xi-\xi_3||\xi+\xi_3|
		\sim  N_{max}^2$.
		\item [$\bullet$] $|\xi_1-\xi_2||\xi_1+\xi_2|= \left(|\xi_1|+|\xi_2|\right)|\xi+\xi_3|
		\sim N_{max}^2$.
	\end{enumerate}
	
	Now, similarly as above, using H\"older's inequality, Lemma \ref{Op de Riesz},  Lemma \ref{imersao dos espacos de Bourgain tipo modulacao}, and Corollary \ref{CorolPPi1} we get
	\begin{equation*}
		\begin{split}
			I & = \sum_{N_1\sim N_2\sim N\gg N_3}
			N^{s+1}\left|\int_{\mathbb{R}^2}(w_1)_{N_1} (w_2)_{N_2} v_{N_3} z_N\, dx dt\right|\\
			& \leq \sum_{N_1\sim N_2\sim N\gg N_3}
			N^{s+1}
			\|(w_1)_{N_1} (w_2)_{N_2}\|_{L^2_{x,t}}
			\|z_N v_{N_3}\|_{L^2_{x,t}}\\
			& \lesssim 
			\sum_{N_1\sim N_2\sim N\gg N_3}
			N_{max}^{s+1}
			\frac{1}{N_{max}}\|(w_1)_{N_1}\|_{X^{0,\frac{1}{2}+\epsilon}_\alpha} 
			\|(w_2)_{N_2}\|_{X^{0,\frac{1}{2}+\epsilon}_\alpha}
			\frac{1}{N}
			\|z_N\|_{X^{0,\frac{1}{2}+\epsilon}}
			\|v_{N_3}\|_{X^{0,\frac{1}{2}+\epsilon}}\\
			& \lesssim 
			\sum_{N_1\sim N_2\sim N\gg N_3}
			N^{-s-\frac2{p}+30\epsilon}_{\max}N_3^{-s+\frac12-\frac1{p}}        
			\|(w_1)_{N_1}\|_{X^{s,\frac{1}{2}+\epsilon}_{p,\alpha}} 
			\|(w_2)_{N_2}\|_{X^{s,\frac{1}{2}+\epsilon}_{p,\alpha}}            
			\|v_{N_3}\|_{X^{s,\frac{1}{2}+\epsilon}_p}\|z_N\|_{X^{0,\frac{1}{2}-2\epsilon}_{p'}}.
		\end{split}
	\end{equation*}

Finally, as in \eqref{semi-resonant case}, applying \eqref{N-epsilon}, we obtain the required estimate \eqref{Estimativ_Bilinearl_Prop_3}.
	\\
	
	\noindent
	\textbf{Case D. Resonant case:} $N_{max}\sim N_{med}\sim N_{min}$ \textbf{:}
	Since $N\sim |\xi|=|\xi_1+\xi_2+\xi_3|\leq |\xi_1|+|\xi_2|+|\xi_3|\lesssim N_1$, in this case, we can assume $N\sim N_1$ because, for $N\ll N_1$ the proof follows as above. So, from now on we consider
	\begin{equation}\label{condition resonant case}
		N\sim N_1\sim N_2\sim N_3\gg1.
	\end{equation}
	
	For further analysis we will use the unit cube decomposition $u=\sum_{n\in\mathbb{Z}} u_n=\sum_{n\in\mathbb{Z}} \Pi_n u$ and for
	the sake of clarity we use the notation $I_n:=[n,n+1)$ with $n\in\mathbb{Z}$. Now, we proceed to derive the required estimate dividing in  several sub-cases.
	\\
	
	\noindent
	\textbf{Subcase D1.} $|\xi_i-\xi_j|\geq |\xi_i+\xi_j|$ \textbf{for some pair} $(i,j) $\textbf{:} Without loss of generality, we will assume $(i,j)=(1,2)$ and proceed by considering five different cases.\\
	
	\noindent
	\textbf{Sub-subcase D1.1.} $|\xi_1+\xi_2|\lesssim 1$ \textbf{and} $\kappa=\min (|\xi_1-\sqrt{\alpha}\, \xi_3|,|\xi_1+\sqrt{\alpha}\, \xi_3|)\lesssim 1$\textbf{:} In order to simplify the notation let $\mathring{\alpha}:=\sqrt{\alpha}$. We will only consider $\kappa=|\xi_1+\mathring{\alpha} \xi_3|\lesssim 1$, the other case is similar.

	Suppose that $\xi_1\in I_n=[n,n+1)$, so that $\xi_1=n+O(1)$,  $\xi_2=-n+O(1)$, $\xi_1+\mathring{\alpha} \xi_3=O(1)$. Therefore $\xi_3=-\frac{n}{\mathring{\alpha}}+O(1)$ and 
	$$
		\xi=-\xi_1-\xi_2-\xi_3=-n-(-n)-(-\frac{n}{\mathring{\alpha}}) +O(1)=\frac{n}{\mathring{\alpha}}+ O(1).
	$$
	For the sake of clarity, we assume $\xi_2\in I_{-n}$, $\xi\in I_n^{\alpha}=[\frac{n}{\mathring{\alpha}}, \frac{n+1}{\mathring{\alpha}})$ ($\R=\cup_n I_n^{\alpha}$) and $\xi_3\in I_{-n}^{\alpha}$.
	Now, using Cauchy Schwartz inequality, (\ref{casos_restantes})  and Proposition~$2.3$ of \cite{Carvajal_and_Panthee} for $s>-\frac{1}{2}$ and $0<\alpha<1$,  we get
	\begin{equation*}
		\begin{split}
			I
			& :=\sum_{n\in\mathbb{Z}}\langle                                        n\rangle^{s+1}\left|\int_{\mathbb{R}^2}(w_1)_n (w_2)_{-n}
			v_{-n}^{\mathring{\alpha}} z_n^{\mathring{\alpha}}\,dxdt\right|\\
			& = \sum_{n\in\mathbb{Z}}\langle n\rangle^{s+1} 
			\left|\int_{\mathbb{R}^2}\mathcal{F}_{x,t}\{(w_1)_n (w_2)_{-n}
			v_{-n}^{\mathring{\alpha}}\}(\xi,\tau) \mathcal{F}_{x,t} \{z_n^{\mathring{\alpha}}\}(\xi,\tau)\,d\xi d\tau\right|\\
			& \leq \sum_{n\in\mathbb{Z}}\langle n\rangle^{s+1} 
			\|(w_1)_n (w_2)_{-n}
			v_{-n}^{\mathring{\alpha}}\|_{X^{0,-\frac{1}{2}+2\epsilon}} \|z_n^{\mathring{\alpha}}\|_{X^{0,\frac{1}{2}-2\epsilon}}\\
			& \sim \sum_{n\in\mathbb{Z}} 
			\|\partial_x((w_1)_n (w_2)_{-n}
			v_{-n}^{\mathring{\alpha}})\|_{X^{s,-\frac{1}{2}+2\epsilon}} 
			\|z_n^{\mathring{\alpha}}\|_{X^{0,\frac{1}{2}-2\epsilon}}\\
			& \lesssim \sum_{n\in\mathbb{Z}}
			\|(w_1)_n\|_{X^{s,\frac{1}{2}+\epsilon}_\alpha}
			\|(w_2)_{-n}\|_{X^{s,\frac{1}{2}+\epsilon}_\alpha}
			\|v_{-n}^{\mathring{\alpha}}\|_{X^{s,\frac{1}{2}+\epsilon}}
			\|z_n^{\mathring{\alpha}}\|_{X^{0,\frac{1}{2}-2\epsilon}}.
		\end{split}
	\end{equation*}
	where $v_{-n}^{\mathring{\alpha}}$ and $z_n^{\mathring{\alpha}}$ are defined as in \eqref{Projetor L-Pg}.
	
	Next, using Holder's inequality twice, we get
	\begin{equation*}
		\begin{split}
			I & \lesssim \left\| \|(w_1)_n\|_{X^{s,\frac{1}{2}+\epsilon}_\alpha}
			\|(w_2)_{-n}\|_{X^{s,\frac{1}{2}+\epsilon}_\alpha}
			\|v_{-n}^{\mathring{\alpha}}\|_{X^{s,\frac{1}{2}+\epsilon}}\right\|_{\ell^p(\mathbb{Z})}
			\left\|\|z_n^{\mathring{\alpha}}\|_{X^{0,\frac{1}{2}-2\epsilon}}\right\|_{\ell^{p'}(\mathbb{Z})}\\
			& \leq \left\|\|(w_1)_n\|_{X^{s,\frac{1}{2}+\epsilon}_\alpha}\right\|_{\ell^{3p}(\mathbb{Z})}
			\left\|\|(w_2)_{-n}\|_{X^{s,\frac{1}{2}+\epsilon}_\alpha}\right\|_{\ell^{3p}(\mathbb{Z})}
			\left\|\|v_{-n}^{\mathring{\alpha}}\|_{X^{s,\frac{1}{2}+\epsilon}}\right\|_{\ell^{3p}(\mathbb{Z})}
			\left\|\|z_n^{\mathring{\alpha}}\|_{X^{0,\frac{1}{2}-2\epsilon}}\right\|_{\ell^{p'}(\mathbb{Z})}\\
			& \sim \|w_1\|_{X^{s,\frac{1}{2}+\epsilon}_{3p,\alpha}}
			\|w_2\|_{X^{s,\frac{1}{2}+\epsilon}_{3p,\alpha}}
			\|v\|_{X^{s,\frac{1}{2}+\epsilon}_{3p}}
			\|z\|_{X^{0,\frac{1}{2}-2\epsilon}_{p'}}\\
			& \leq \|w_1\|_{X^{s,\frac{1}{2}+\epsilon}_{p,\alpha}}
			\|w_2\|_{X^{s,\frac{1}{2}+\epsilon}_{p,\alpha}}
			\|v\|_{X^{s,\frac{1}{2}+\epsilon}_{p}}
			\|z\|_{X^{0,\frac{1}{2}-2\epsilon}_{p'}}.
		\end{split}
	\end{equation*}
	
	\noindent
	\textbf{Sub-subcase D1.2.} $|\xi_1+\xi_2|\lesssim 1$ \,\textbf{and}\, $|\xi_1\pm\mathring{\alpha}\, \xi_3|\gg 1$\textbf{:}  
	In this sub-case we assume that $\xi_1\in I_n$ and $\xi_3\in I_m$, so that $\xi_1=n+O(1)$, $\xi_2=-n+O(1)$, $\xi_3=m+O(1)$ and $\xi=-\xi_1-\xi_2-\xi_3= -n-(-n)-m+O(1)=-m+O(1)$.
	
	Again, for simplicity, we estimate only the case when \,$\xi_2\in I_{-n}$ and $\xi\in I_{-m}$.
	From Proposition~\ref{O(1)}, we have $|\xi_1\pm\alpha \xi_3|\sim |n\pm \alpha m|\gg 1$. Also, recall that from 
	(\ref{casos_restantes}) and (\ref{condition resonant case}), in this case, $|m|\sim|n|\gg 1$.\\
	Without loss of generality, we can assume $|\xi_1+\mathring{\alpha} \xi_3|\geq |\xi_1-\mathring{\alpha} \xi_3|$.
	Using Holder's inequality, Corollary
	\ref{CorolPPi1}  twice, (\ref{casos_restantes}), \,and\, Proposition \ref{O(1)}, we obtain:
	\begin{equation}\label{sub subcase 4.1.2}
		\begin{split}
			\sum_{\begin{smallmatrix}
				|m|\sim |n|\gg 1\\|n \pm \mathring{\alpha} m| \gg  1
		\end{smallmatrix}} 
		|m|^{s+1}&\left| \int_{\mathbb{R}^2} (w_1)_n (w_2)_{-n} v_{m} z_{-m}\, dxdt\right|
		\leq
		\sum_{\begin{smallmatrix}
				|m|\sim |n|\gg 1\\ |n \pm \mathring{\alpha} m| \gg  1
		\end{smallmatrix}}
		|m|^{s+1}\|(w_1)_n v_m\|_{L^2_{x,t}} \|(w_2)_{-n} z_{-m}\|_{L^2_{x,t}}
		\\
		& \lesssim 
			\sum_{\begin{smallmatrix}
					|m|\sim|n|\gg 1\\|n \pm \mathring{\alpha} m| \gg  1
			\end{smallmatrix}}
			\frac{|n|^{s+1}}{|n+\mathring{\alpha} m||n-\mathring{\alpha} m|}
			\|(w_1)_n \|_{X_\alpha^{0,\frac{1}{2}+}}
			\|v_m \|_{X^{0,\frac{1}{2}+}}
			\|(w_2)_{-n}\|_{X_\alpha^{0,\frac{1}{2}+}}
			\|z_{-m} \|_{X^{0,\frac{1}{2}+}}\\
			& \lesssim 
			\sum_{\begin{smallmatrix}
					|m|\sim |n|\gg 1\\|n \pm \mathring{\alpha} m| \gg  1
			\end{smallmatrix}}
			\frac{|n|^{-2s+1+\epsilon}}{|n \pm \mathring{\alpha} m|(|n|+|\mathring{\alpha} m|)}
			\|(w_1)_n\|_{X_\alpha^{s,\frac{1}{2}+}}
			\|v_m\|_{X^{s,\frac{1}{2}+}}
			\|(w_2)_{-n}\|_{X_\alpha^{s,\frac{1}{2}+}}
			\|z_{-m} \|_{X^{0,\frac{1}{2}-}}\\
			& \sim
			\sum_{\begin{smallmatrix}
					|m|,|n|\gg 1\\|n \pm \mathring{\alpha} m| \gg  1
			\end{smallmatrix}}
			\frac{1}{|n-\mathring{\alpha} m| |n|^{2s-\epsilon}}
			\|(w_1)_n \|_{X_\alpha^{s,\frac{1}{2}+}}
			\|v_m \|_{X^{s,\frac{1}{2}+}}
			\|(w_2)_{-n}\|_{X_\alpha^{s,\frac{1}{2}+}}
			\|z_{-m}\|_{X^{0,\frac{1}{2}-}}\\
			& =
			\sum_{\begin{smallmatrix}
					|m|,|n|\gg 1\\|n \pm \mathring{\alpha} m| \gg  1
			\end{smallmatrix}}
			\frac{1}{|n-\mathring{\alpha} m|\langle n\rangle^{2s-\epsilon}}
			a_n
			b_n
			c_m
			d_m
		\end{split}
	\end{equation}
	where $a_n=\|(w_1)_n \|_{X_\alpha^{s,\frac{1}{2}+}}$, $b_n=\|(w_2)_{-n}\|_{X_\alpha^{s,\frac{1}{2}+}}$, $c_m=\|v_m \|_{X^{s,\frac{1}{2}+}}$, and $d_m=\|z_{-m}\|_{X^{0,\frac{1}{2}-}}$.

	Notice that
	\begin{equation}\label{x1sub subcase 4.1.2xw}
		\begin{split}
	\sum_{\begin{smallmatrix}
					|m|,|n|\gg 1\\|n \pm \mathring{\alpha} m| \gg  1
			\end{smallmatrix}}
			\frac{1}{|n-\mathring{\alpha} m|\langle n\rangle^{2s-\epsilon}}
			a_n
			b_n
			c_m
			d_m&\leq \sum_{|n-\mathring{\alpha} m|\geq |n|\geq1}\frac{1}{|n-\mathring{\alpha} m|\langle n\rangle^{2s-\epsilon}}
			a_n
			b_n
			c_m
			d_m\\
			&\quad +\sum_{1 \leq |n-\mathring{\alpha} m|< |n|}\frac{1}{|n-\mathring{\alpha} m|\langle n\rangle^{2s-\epsilon}}
			a_n
			b_n
			c_m
			d_m\\
			&
			=:T_1+T_2
		\end{split}
	\end{equation}
	
	To estimate $T_1$, we use Holder's inequality, to get
	\begin{equation}\label{x1sub subcase 4.1.2xw3}
		\begin{split}
	T_1\leq \sum_{ |n-\mathring{\alpha} m|\geq |n|\geq 1}\frac{1}{\langle n\rangle^{1+2s-\epsilon}}
			a_n
			b_n
			c_m
			d_m&\leq \sum_n \frac{1}{\langle n\rangle^{1+2s-\epsilon}} a_n b_n\sum _m c_m d_m\\
			&\leq \sum_n \frac{1}{\langle n\rangle^{1+2s-\epsilon}} a_n b_n\|c_m\|_{\ell^p}\|d_m\|_{\ell^{p'}}\\
			&\leq \|a_n\|_{\ell^p}\left\|\frac{b_n}{\langle n\rangle^{1+2s-\epsilon}}\right\|_{\ell^{p'}} \|c_m\|_{\ell^p}\|d_m\|_{\ell^{p'}}\\
			&\leq \|a_n\|_{\ell^p}\|b_n\|_{\ell^p}\left\|\frac{1}{\langle n\rangle^{1+2s-\epsilon}}\right\|_{\ell^{\frac{p}{p-2}}} \|c_m\|_{\ell^p}\|d_m\|_{\ell^{p'}}\\
			&\lesssim
			\|a_n\|_{\ell^p}\|b_n\|_{\ell^p} \|c_m\|_{\ell^p}\|d_m\|_{\ell^{p'}},
	\end{split}
	\end{equation}
	where the condition $\frac{(1+2s)p}{p-2}>1$ for  $p>2$ is used, which is indeed true since $s>\frac14-\frac{1}{p}>-\frac{1}{p}$. The case $p=2$ is similar considering $\ell^{\infty}$ instead of $\ell^{\frac{p}{p-2}}$. 
	
	Similarly, to estimate $T_2$, we use Holder's inequality, Minkowsky inequality and inequality \eqref{sumcomreal}, and obtain
	\begin{equation}\label{x1sub subcase 4.1.2xw4}
		\begin{split}
	T_2&\leq \sum_{ 1\leq |n-\mathring{\alpha} m|< |n|}\frac{1}{\langle  n-\mathring{\alpha}a m\rangle^{1+2s-\epsilon}}
			a_n
			b_n
			c_m
			d_m\\
			&\leq \sum_n  a_n b_n\sum _m \frac{c_m d_m}{\langle n-\mathring{\alpha} m\rangle^{1+2s-\epsilon}}\\
			&\leq \|a_n\|_{\ell^p}\left\|b_n \sum _m \frac{c_m d_m}{\langle n-\mathring{\alpha} m\rangle^{1+2s-\epsilon}}\right\|_{\ell^{p'}_n} \\
			&\leq \|a_n\|_{\ell^p}\|b_n\|_{\ell^p} \left\|\sum _m \frac{c_m d_m}{\langle n-\mathring{\alpha} m\rangle^{1+2s-\epsilon}}\right\|_{\ell^{\frac{p}{p-2}}_n }\\
			&\leq \|a_n\|_{\ell^p}\|b_n\|_{\ell^p} \sum_m |c_m d_m|\left\|\frac{1}{\langle n-\mathring{\alpha} m\rangle^{1+2s-\epsilon}}\right\|_{\ell^{\frac{p}{p-2}}_n }\\
			&\lesssim
			\|a_n\|_{\ell^p}\|b_n\|_{\ell^p} \|c_m\|_{\ell^p}\|d_m\|_{\ell^{p'}}.
	\end{split}
	\end{equation}
\\

\noindent
	\textbf{Sub-subcase D1.3.} $|\xi_1+\xi_2|\gg 1$ \,\textbf{and}\, $|\xi_i-\xi_j|\lesssim 1$ \textbf{for some} $(i,j)\neq (1,2)$: Without loss of generality, we may assume $(i,j)=(1,3)$. We will prove the required estimate in considering  $|\xi-\xi_3|\lesssim 1$  and $|\xi-\xi_3|\gg 1$ separately.\\

	First, consider $|\xi-\xi_3|\lesssim 1$: Let $\xi_1\in I_n$ so that $\xi_1=n+O(1)$, and thus $\xi_3=n+O(1)$. Therefore, $|\xi-n|=|\xi-\xi_3+\xi_3-n|\leq |\xi-\xi_3|+|\xi_3-n|\lesssim 1$. Consequently, $\xi=n+O(1)$ and $\xi_2=-\xi_1-\xi_3-\xi=-3n+O(1)$.\\
	For clarity, we assume $\xi\in I_n$, $\xi_2\in I_{-3n}$ and $\xi_3\in I_n$.
	Now, similarly as in the \textbf{Sub-subcase D1.1}, we obtain
	\begin{equation*}
		\begin{split}
			I 
			& =\sum_{n\in\mathbb{Z}}\langle n\rangle^{s+1}\left|\int_{\mathbb{R}^2}(w_1)_n (w_2)_{-3n} v_{n} z_n\,dxdt\right|\\
			& \leq \|w_1\|_{X^{s,\frac{1}{2}+\epsilon}_{p,\alpha}}
			\|w_2\|_{X^{s,\frac{1}{2}+\epsilon}_{p,\alpha}}
			\|v\|_{X^{s,\frac{1}{2}+\epsilon}_{p}}
			\|z\|_{X^{0,\frac{1}{2}-2\epsilon}_{p'}}.
		\end{split}
	\end{equation*}

	Now, consider $|\xi-\xi_3|\gg 1$: Let $\xi_1\in I_n$ and $\xi\in I_m$ so that $\xi_1=n+O(1)$, $\xi=m+O(1)$, $\xi_3=n+O(1)$ and $\xi_2=-\xi_1-\xi_3-\xi=-2n-m+O(1)$.
	\\
	
	We will suppose that $|\xi-\xi_3|\geq |\xi+\xi_3|$, the case $|\xi+\xi_3|\geq |\xi-\xi_3|$ follows similarly. Since $\xi-\xi_3=m-n+O(1)$, using the Proposition \ref{O(1)} and Proposition \ref{O(2)}, we obtain 
	\begin{equation}\label{4.1.3.2_1}
		|m-n|\sim |\xi-\xi_3|=\max\{|\xi\pm                 
		\xi_3|\}=|\xi|+|\xi_3|\sim |m|\sim |n|.
	\end{equation}
	Similarly, since $\xi_1-\xi_2=3n+m+O(1)$, one has
	\begin{equation*}
		|3n+m|\sim |\xi_1-\xi_2|=\max\{|\xi_1\pm\xi_2|\}=|\xi_1|+|\xi_2|\sim |m|\sim |n|.
	\end{equation*}
	As $\xi_1+\xi_2=-m-n+O(1)$, again by  Proposition \ref{O(1)}, we get that $|m+n|\sim |\xi_1+\xi_2|\gg 1$.
	\\
	As above we only consider the contribution when $\xi_3\in I_n$ and  $\xi_2\in I_{-2n-m}$.
	Using H\"older's inequality, Corollary \ref{Pi_mnalpha} \,two times, (\ref{casos_restantes}), (\ref{4.1.3.2_1}), we have
	\begin{equation*}
		\begin{split}
			I &:=\sum_{\begin{smallmatrix}
					m,n\in\mathbb{Z}\\|m|\sim |n|, |m+n|\gg 1\\
					|m+3n|,|m-n|\gtrsim |n|
			\end{smallmatrix}} 
		|m|^{s+1}\left| \int_{\mathbb{R}^2} (w_1)_n (w_2)_{-2n-m} v_{n} z_{m}\, dxdt\right|\\
		&\leq
			\sum_{\begin{smallmatrix}
					m,n\in\mathbb{Z}\\|m|\sim |n|, |m+n|\gg 1\\
					|m+3n|,|m-n|\gtrsim |n|
			\end{smallmatrix}}
			|n|^{s+1}\|(w_1)_n (w_2)_{-2n-m}\|_{L^2_{x,t}} 
			\|v_{n} z_{m}\|_{L^2_{x,t}}\\ 
			&\lesssim 
			\sum_{\begin{smallmatrix}
					m,n\in\mathbb{Z}\\|m|\sim |n|, |m+n|\gg 1\\
					|m+3n|,|m-n|\gtrsim |n|
			\end{smallmatrix}}
			\frac{|n|^{s+1}}{|m+n|\sqrt{|m-n||m+3n|}}
			\|(w_1)_n \|_{X^{0,\frac{1}{2}+}_\alpha}
			\|(w_2)_{-2n-m}\|_{X^{0,\frac{1}{2}+}_\alpha}
			\|v_n\|_{X^{0,\frac{1}{2}+}}
			\|z_m\|_{X^{0,\frac{1}{2}+}}\\
			& \lesssim 
			\sum_{\begin{smallmatrix}
					m,n\in\mathbb{Z}\\|m|\sim |n|, |m+n|\gg 1\\
					|m+3n|,|m-n|\gtrsim |n|
			\end{smallmatrix}}
			\frac{1}{|m+n|\, |n|^{2s-}}
			\|(w_1)_n\|_{X^{s,\frac{1}{2}+}_\alpha}
			\|(w_2)_{-2n-m}\|_{X^{s,\frac{1}{2}+}_\alpha}
			\|v_n\|_{X^{s,\frac{1}{2}+}}
			\|z_m\|_{X^{0,\frac{1}{2}-}}\\
			& =
			\sum_{\begin{smallmatrix}
					|m|, |n|, |m+ n|\gg 1
			\end{smallmatrix}}
			\frac{1}{|m+n|\, |n|^{2s-}}
			a_n b_n  c_{-2n-m} d_m\\
			& \lesssim
			\sum_{\begin{smallmatrix}
					|m+ n|\geq |n|\geq 1
			\end{smallmatrix}}
			\frac{1}{|n|^{1+2s-}}
			a_n b_n  c_{-2n-m}d_m+ \sum_{\begin{smallmatrix}
					1\leq |m+ n|< |n|
			\end{smallmatrix}}
			\frac{1}{|m+n|^{1+2s-}}
			a_n b_n  c_{-2n-m}d_m,
		\end{split}
	\end{equation*}
	the rest of the proof is the same as in \eqref{x1sub subcase 4.1.2xw3} and \eqref{x1sub subcase 4.1.2xw4}, so we omit the details.
	\\
	
	\noindent
	\textbf{Sub-subcase D1.4.} $|\xi_1+\xi_2|\gg 1$ \,\textbf{and}\, $|\xi_i+\xi_j|\lesssim 1$ \textbf{for some} $(i,j)\neq (1,2)$ \textbf{:}
	In this case we follow the same arguments used in \textbf{Sub-subcase D1.3.}
	\\
	
	\noindent
	\textbf{Sub-subcase D1.5.} $|\xi_1+\xi_2|\gg 1$ \,\textbf{and}\, $|\xi_i+\xi_j|\gg 1$ \textbf{for some} $(i,j)\neq (1,2)$ \textbf{:}
	In the \textbf{Subcase D1} we considered $|\xi_1-\xi_2|\geq |\xi_1+\xi_2|\gg 1$, and consequently we have $|\xi_i\pm\xi_j|\gg 1$ for all $(i,j)$. Moreover $|\xi+\xi_j|=|-\xi_1-\xi_2-\xi_3+\xi_j|=\left|\sum_{i\in \{1,2,3\}\backslash\{j\}}\xi_i\right|\gg 1, j=1,2,3$.  
	For $\tau+\tau_1+\tau_2+\tau_3=0$ and $\xi+\xi_1+\xi_2+\xi_3=0$, we have
	\begin{equation}\label{sigmas_sumatorio}
\begin{split}
	\sigma+\sigma^1+\sigma^2+\sigma_3&=(\tau-\xi^3)+(\tau_1- \xi_1^3)+(\tau_2- \xi_2^3)+(\tau_3- \xi_3^3)\\
&=-\xi^3- \xi_1^3-\xi_2^3- \xi_3^3
\\
& =3(\xi_1+\xi_2)(\xi_2+\xi_3)(\xi_3+\xi_1)
\end{split}
	\end{equation}
where $\sigma^j=\sigma^j(\tau_j,\xi_j)=\tau-\xi^3_j$, $j=1,2$.
Let
 \begin{equation*}
		\sigma^{max}:=\max\{|\sigma|,|\sigma^1|,|\sigma^2|,|\sigma_3|\}.
	\end{equation*}

	Without loss of generality, we assume that $|\sigma^1|=|\sigma^{max}|$.  Let $\xi_j\in I_{n_j}$, $j=1,2,3$ and $\xi\in I_n$.
	Using  Proposition \ref{O(1)}, one obtains
	$$\langle\sigma^1\rangle\gtrsim|\xi_1+\xi_2||\xi_2+\xi_3|
	|\xi_1+\xi_3|\sim|n_1+n_2||n_2+n_3||n_1+n_3|\gg 1.$$
	Observe that 
	\begin{equation}\label{4.1.5_1_2}
		n_1+n_2+n_3+n\leq 0=\xi+\xi_1+\xi_2+\xi_3\leq n_1+n_2+n_3+n+4.
	\end{equation}

	Let\, $\overset{\alpha}{f}(x):=\alpha^{1/3}f(\alpha^{1/3}x)$. With this definition, we have
	\begin{equation}\label{4.1.5_2_2}
		\|w_{n_2}\|_{L^\infty_t L^2_x}\sim\|\overset{\alpha}{w_{n_2}}\|_{L^\infty_t L^2_x}
		\lesssim\|\overset{\alpha}{w_{n_2}}\|_{X^{0,\frac{1}{2}+}}\sim\|w_{n_2}\|_{X_{\alpha}^{0,\frac{1}{2}+}},
\end{equation}
and 
\begin{equation}\label{4.1.5_2_2x}
\begin{split}
		\|w_{n_1}\|_{L^2_{x,t}} &\sim \|\overset{\alpha}{w_{n_1}}\|_{L^2_{x,t}}\lesssim\frac{1}{\sqrt{|n_1+n_2|
				|n_2+n_3||n_1+n_3|}}
		\|\overset{\alpha}{w_{n_1}}\|_{X^{0,\frac{1}{2}+}} \\
&\sim\frac{1}{\sqrt{|n_1+n_2|
				|n_2+n_3||n_3+n_1|}}\|w_{n_1}\|_{X_\alpha^{0,\frac{1}{2}+}}.
\end{split}
	\end{equation}
	From \eqref{4.1.5_1_2}, \eqref{4.1.5_2_2},  \eqref{4.1.5_2_2x},  Bernstein and Holder's inequalities and Corollary \ref{Pi_mnalpha}, we have
	\begin{equation*}
		\begin{split}
			I
			&:=\sup_{\begin{smallmatrix}
					k\in \Z\\|k| \gg  1
			\end{smallmatrix}}\sum_{\begin{smallmatrix}|n|, |n_1|, |n_2|, |n_3|\sim |k|\\
			n_1+n_2+n_3+n=O(1)
			\end{smallmatrix}} |n|^{s+1}
			\left|\int_{\mathbb{R}^2}w_{n_1}w_{n_2}v_{n_3}z_n\,         dxdt\right|\\
			& \leq \sup_{\begin{smallmatrix}
					k\in \Z\\|k| \gg  1
			\end{smallmatrix}}\sum_{\begin{smallmatrix}|n|, |n_1|, |n_2|, |n_3|\sim |k|\\
			n_1+n_2+n_3+n=O(1)
			\end{smallmatrix}}|n|^{s+1}
			\|w_{n_1}w_{n_2}\|_{L^2_{x,t}}
			\|v_{n_3}z_n\|_{L^2_{x,t}}\\
			& \lesssim \sup_{\begin{smallmatrix}
					k\in \Z\\|k| \gg  1
			\end{smallmatrix}}\sum_{\begin{smallmatrix}|n|, |n_1|, |n_2|, |n_3|\sim |k|\\
			n_1+n_2+n_3+n=O(1)
			\end{smallmatrix}}|n|^{s+1}
			\|w_{n_1}\|_{L^2_{x,t}}
			\|w_{n_2}\|_{L^\infty_t L^2_x}
			\|v_{n_3}z_n\|_{L^2_{x,t}}.
		\end{split}
	\end{equation*}
	
Now,
	\begin{equation*}
		\begin{split}
			I
			&\lesssim \sup_{\begin{smallmatrix}
					k\in \Z\\|k| \gg  1
			\end{smallmatrix}}\sum_{\begin{smallmatrix}|n|, |n_1|, |n_2|, |n_3|\sim |k|\\
			n_1+n_2+n_3+n=O(1)
			\end{smallmatrix}}|n|^{s+1+}
			\frac{1}{\sqrt{|n_1+n_2|
					|n_2+n_3||n_3+n_1|}}
			\frac{1}{\sqrt{|n_3+n||n_3-n|}}\times\\
			& \hspace{0.5cm}\times
			\|w_{n_1}\|_{X_\alpha^{0,\frac{1}{2}+}}
			\|w_{n_2}\|_{X_\alpha^{0,\frac{1}{2}+}}
			\|v_{n_3}\|_{X^{0,\frac{1}{2}+}}
			\|z_n\|_{X^{0,\frac{1}{2}-}}.
		\end{split}
	\end{equation*}
	Using Proposition \ref{O(1)}, we obtain
	\begin{enumerate}
		\item [a)]$|n_1-n_2|\sim|\xi_1-\xi_2|=\max\{|\xi_1\pm\xi_2|\}=
		|\xi_1|+|\xi_2|\sim |n|$, 
		\item [b)] $n_{03}=\max\{|n\pm n_3|\}=|n|+|n_3|\sim|n|$,
		\item [c)] $n_{13}=\max\{|n_1+n_3|,|n_2+n_3|\}=
		\frac{|n_1+2n_3+n_2|}{2}+\frac{|n_1-n_2|}{2}
		\gtrsim |n|$.
	\end{enumerate}
	Without loss of generality, we assume that $n_{03}=|n_3-n|\gtrsim |n|$ \,and\, $n_{13}=|n_1+n_3|\gtrsim |n|$. Consequently, from a), b), c) above and  (\ref{sub subcase 4.1.2}), we get
	\begin{equation*}
		\begin{split}
			I
			&\lesssim \sup_{\begin{smallmatrix}
					k\in \Z\\|k| \gg  1
			\end{smallmatrix}}|k|\sum_{\begin{smallmatrix}|n_1|, |n_2|, |n_3|\sim |k|\\
			n_1+n_2+n_3+n=O(1)
			\end{smallmatrix}} \frac{a_{n_1} b_{n_2}c_{n_3}d_{n}}{|n+n_3||n_2+n_3|^{1/2}|n|^{1+2s-}}\\
			&=:\sup_{\begin{smallmatrix}
					k\in \Z\\|k| \gg  1
			\end{smallmatrix}}|k| \quad \mathcal{X}(k)
		\end{split}	
	\end{equation*}
	where $a_{n_1} =\|w_{n_1}\|_{X^{s,\frac{1}{2}+}_\alpha}$, $b_{n_2}=\|w_{n_2}\|_{X^{s,\frac{1}{2}+}_\alpha}$, $c_{n_3}=\|v_{n_3}\|_{X^{s,\frac{1}{2}+}}$, $d_{n}=\|z_n\|_{X^{0,\frac{1}{2}-}}$. Considering $a_{n_1}$ as $a_{-n-\gamma}$, $\gamma \in \Z$ and considering the sum on $n_2$, $n_3$ and $n$, we have two cases, viz.,  $|n+n_3|\geq |n|$ and $|n+n_3|\leq |n|$.

	First, consider the case $|n+n_3|\geq |n|$. In this case, using H\"older's and Young inequalities
	\begin{equation*}
		\begin{split}
			\mathcal{X}(k)
			&\lesssim \|b_{n_2}\|_{\ell^p}\left\|\sum_{|n_3|\sim |k|} \frac{ c_{n_3}}{|n_3|^{2+2s-}|n_2+n_3|^{1/2}}\sum_{n} a_{-n -\gamma}d_n\right\|_{\ell^{p'}(|n_2|\sim |k|)}\\
			&\lesssim \|a_{n_1}\|_{\ell^p}\|b_{n_2}\|_{\ell^p}\|d_{n}\|_{\ell^{p'}}\left\|\frac{c_{n_3}}{|n_3|^{2+2s-}}\right\|_{\ell^{r_1}(|n_3|\sim |k|)}\left\|\frac{1}{|n_2|^{1/2}}\right\|_{\ell^{r_2}(|n_2|\sim |k|)}\\
			&\lesssim \|a_{n_1}\|_{\ell^p}\|b_{n_2}\|_{\ell^p}\|c_{n_3}\|_{\ell^p}\|d_{n}\|_{\ell^{p'}}\left\|\frac{1}{|n_3|^{2+2s-}}\right\|_{\ell^{r_3}(|n_3|\sim |k|)}\left\|\frac{1}{|n_2|^{1/2}}\right\|_{\ell^{r_2}(|n_2|\sim |k|)}\\
&\lesssim \|a_{n_1}\|_{\ell^p}\|b_{n_2}\|_{\ell^p}\|c_{n_3}\|_{\ell^p}\|d_{n}\|_{\ell^{p'}}\frac{1}{|k|^{(2+2s-)-1/r_3}}\frac{1}{|k|^{1/2-1/r_2}}
		\end{split}	
	\end{equation*}
	where 
	\begin{equation}\label{indices1}
	\frac{1}{p'}+1=\frac{1}{r_1}+\frac{1}{r_2}, \quad \textrm{and} \quad \frac{1}{r_1}=\frac{1}{p}+\frac{1}{r_3}.
	\end{equation}
	
	The relations in \eqref{indices1} and the relation $s>\frac14-\frac1{p}$ imply that
	$$
	{(2+2s-)+\frac12-\frac1{r_2}-\frac1{r_3}}>1.
	$$

	Now, consider the case  $|n+n_3|\leq |n|$. In this case too, using H\"older's and Young inequalities, we get
	\begin{equation}\label{4.1.5_0x}
		\begin{split}
			\mathcal{X}(k)
			&\lesssim \|b_{n_2}\|_{\ell^p}\left\|\sum_{|n_3|\sim |k|} \frac{ c_{n_3}}{|n_2+n_3|^{1/2}}\sum_{|n|\sim |k|} \frac{a_{-n -\gamma}d_n}{|n+n_3|^{2+2s-}}\right\|_{\ell^{p'}(|n_2|\sim |k|)}.
	\end{split}	
	\end{equation}
	
	We estimate \eqref{4.1.5_0x} dividing in the following two cases:
	
	\noindent
	{\bf i) $|n+n_3|\geq |n_2+n_3|$:}
	In this case we have
	\begin{equation*}
		\begin{split}
			\mathcal{X}(k)
			&\lesssim \|a_{n_1}\|_{\ell^p}\|b_{n_2}\|_{\ell^p}\|d_{n}\|_{\ell^{p'}}\left\|\sum_{|n_3|\sim |k|} \frac{ c_{n_3}}{|n_2+n_3|^{5/2+2s-}}\right\|_{\ell^{p'}(|n_2|\sim |k|)}\\
			&\lesssim \|a_{n_1}\|_{\ell^p}\|b_{n_2}\|_{\ell^p}\|c_{n_3}\|_{\ell^p}\|d_{n}\|_{\ell^{p'}}\left\|\sum_{|n_3|\sim |k|} \frac{1}{|n_3|^{5/2+2s-}}\right\|_{\ell^{\frac{p}{(2(p-1)}}(|n_2|\sim |k|)}\\
			&\lesssim \|a_{n_1}\|_{\ell^p}\|b_{n_2}\|_{\ell^p}\|c_{n_3}\|_{\ell^p}\|d_{n}\|_{\ell^{p'}}\frac{1}{|k|^{\frac52+2s-\frac{2(p-1)}{p}}}, 
	\end{split}	
	\end{equation*}
	and $\frac52+2s-\frac{2(p-1)}{p}>1$ since $s>\frac14-\frac1{p}$.\\

	\noindent
	{\bf ii) $|n+n_3|\leq |n_2+n_3|$: }	
	In this case using H\"older's and Young inequalities, we have
	\begin{equation*}
		\begin{split}
			\mathcal{X}(k)
			&\lesssim \|b_{n_2}\|_{\ell^p}\left\|\sum_{|n_3|\sim |k|} c_{n_3} \sum_{|n|\sim |k|} \frac{a_{-n -\gamma}d_n}{|n+n_3|^{5/2+2s-}}\right\|_{\ell^{p'}(|n_2|\sim |k|)}\\
		&\lesssim |k|^{1/p'}\|b_{n_2}\|_{\ell^p} \|c_{n_3}\|_{\ell^p}\left\|\sum_{|n|\sim |k|} \frac{a_{-n -\gamma}d_n}{|n+n_3|^{5/2+2s-}}\right\|_{\ell^{p'}(|n_3|\sim |k|)}\\	
		&\lesssim |k|^{1/p'}\|b_{n_2}\|_{\ell^p} \|c_{n_3}\|_{\ell^p}\|a_{-n -\gamma}d_n\|_{\ell^1}\left\|\sum_{|n|\sim |k|} \frac{1}{|n_3|^{5/2+2s-}}\right\|_{\ell^{p'}(|n_3|\sim |k|)}\\	
		&\lesssim |k|^{1/p'}\|a_{n_1}\|_{\ell^p}\|b_{n_2}\|_{\ell^p} \|c_{n_3}\|_{\ell^p}\|d_n\|_{\ell^{p'}}\frac{1}{|k|^{\frac52+2s-\frac{1}{p'}}},
	\end{split}	
	\end{equation*}
	and $\frac52+2s-\frac{1}{p'}>\frac{1}{p'}+1$ if $s>\frac14-\frac1{p}$.\\
	
	\noindent
	\textbf{Subcase D2.} $|\xi_i-\xi_j|\leq |\xi_i+\xi_j|$ \textbf{for all pair} $(i,j)$ \textbf{:} In this case, by Proposition \ref{O(1)}, for all $j=1,2,3$, $\xi_j$ has the same sign and $|\xi_i+\xi_j|=|\xi_i|+|\xi_j|$. Furthermore since $\xi+\xi_1+\xi_2+\xi_3=0$, we get
	\begin{equation*}
		|\xi+\xi_j|=\left|\sum_{i\in\{1,2,3\}\setminus\{j\}} \xi_i\right| \sim |\xi_{max}|,\,\,|\xi-\xi_j|=|\xi|+|\xi_j|\sim |\xi_{max}|.
	\end{equation*}
	From (\ref{condition resonant case}) and (\ref{sigmas_sumatorio}), we also have 
	\begin{equation}\label{Subcase 4.2}
		\langle \sigma_{max}\rangle \gtrsim |\xi_{max}|^3.
	\end{equation}
	
	\noindent
	\textbf{Sub-subcase D2.1.} \,$\sigma_j=\sigma_{\max}$ \textbf{for some} $j=1,2,3$:  In this case we use a dyadic descomposition. We assume  $\sigma_1=\sigma_{\max}$ and can proceed similarly in the other cases.
	
	Using Bernstein's inequality, (\ref{Subcase 4.2}), Lemma \ref{Op de Riesz} $(|\xi+\xi_3|^\frac{1}{2} |\xi-\xi_3|^\frac{1}{2}\gtrsim N)$ together with the inequality (\ref{casos_restantes}) $(\langle\sigma\rangle^{\frac{1}{2}+}\lesssim \langle\sigma\rangle^{\frac{1}{2}-} N^{0+})$,  similarly as in the non-resonant case, we have that
	\begin{equation*}
		\begin{split}
			I
			& := \sum_{N_{min}\sim N_{med}\sim N} N^{s+1} \left|\int_{\mathbb{R}^2}(w_1)_{N_1} (w_2)_{N_2} v_{N_3} z_N \,dxdt\right|\\
			& \lesssim \sum_{N_{min}\sim N_{med}\sim N} N^{s+1+\frac{1}{2}-1}
			\|(w_1)_{N_1}\|_{L^2_{x,t}} \|(w_2)_{N_2}\|_{L^2_x L^\infty_t} \|v_{N_3}\|_{X^{0,\frac{1}{2}+}} \|z_N\|_{X^{0,\frac{1}{2}+}}\\
			& \lesssim \sum_{N_{min}\sim N_{med}\sim N} N^{s+\frac{1}{2}}
			N_1^{-3(\frac{1}{2}+\epsilon)}
			\|(w_1)_{N_1}\|_{X^{0,\frac{1}{2}+}_\alpha} 
			\|(w_2)_{N_2}\|_{X^{0,\frac{1}{2}+}_\alpha} 
			\|v_{N_3}\|_{X^{0,\frac{1}{2}+}} \|z_N\|_{X^{0,\frac{1}{2}+}}\\
			& \lesssim \sum_{N_{min}\sim N_{med}\sim N} N^{s-1-3s+3(\frac12-\frac1{p})+}
			\|(w_1)_{N_1}\|_{X^{s,\frac{1}{2}+}_{p,\alpha}} 
			\|(w_2)_{N_2}\|_{X^{s,\frac{1}{2}+}_{p,\alpha}} 
			\|v_{N_3}\|_{X^{s,\frac{1}{2}+}_p} \|z_N\|_{X^{0,\frac{1}{2}+}_p}
		\end{split}
	\end{equation*}
	and $s-1-3s+3(\frac12-\frac1{p})+<0$ if $s>\frac14-\frac3{2p}$.
	\\
	
	\noindent
	\textbf{Sub-subcase D2.2.} $\sigma=\sigma_{\max}$\textbf{:} In this case we will use the unit-cube decomposition. We proceed by dividing this into further two sub-cases:
	
	\noindent
	\textbf{Sub-subcase D2.2.1.} $\sigma=\sigma_{\max}$ \textbf{and} $|\xi_i-\xi_j|\lesssim 1$ \textbf{for some} $i\neq j$ \textbf{:} Without loss of generality, we may assume $|\xi_1-\xi_2|\lesssim 1$ and analyse considering $|\xi_1-\xi_3|\lesssim 1$ and  $|\xi_1-\xi_3|\gg 1$ separately.
	\\
	
	First consider the case when $|\xi_1-\xi_3|\lesssim 1$. In this case, let $\xi_1\in I_n$, so that $\xi_2=n+O(1)$, $\xi_3=n+O(1)$ and $\xi=-3n+O(1)$.
	\\
	Considering only the diagonal case, and using the H\"older's and Bernstein's inequalities along with (\ref{Subcase 4.2}), we have that
	\begin{equation}\label{4.2.2.1}
		\begin{split}
			\sum_{n\in\mathbb{Z}} |n|^{s+1} &\left|\int_{\mathbb{R}^2}(w_1)_n (w_2)_n v_n z_{-3n}\,dxdt\right| \leq \sum_{n\in\mathbb{Z}} |n|^{s+1} \|(w_1)_n (w_2)_n\|_{L^2_{x,t}} \|v_n z_{-3n}\|_{L^2_{x,t}}\\
			& \leq \sum_{n\in\mathbb{Z}} |n|^{s+1}  \|(w_1)_n\|_{L^\infty_{x,t}}
			\|(w_2)_n\|_{L^2_{x,t}} \|v_n\|_{L^\infty_{x,t}} \|z_{-3n}\|_{L^2_{x,t}}\\
			& \lesssim \sum_{n\in\mathbb{Z}} |n|^{s+1}
			\|(w_1)_n\|_{X^{0,\frac{1}{2}+}_\alpha}
			\|(w_2)_n\|_{X^{0,\frac{1}{2}+}_\alpha} 
			\|v_n\|_{X^{0,\frac{1}{2}+}} 
			|n|^{-\frac{3}{2}+3\epsilon}\|z_{-3n}\|_{X^{0,\frac{1}{2}-}}\\
			& \sim \sum_{n\in\mathbb{Z}} \langle n\rangle^{-2s-\frac{1}{2}+3\epsilon}
			\|(w_1)_n\|_{X^{s,\frac{1}{2}+}_\alpha}
			\|(w_2)_n\|_{X^{s,\frac{1}{2}+}_\alpha} 
			\|v_n\|_{X^{s,\frac{1}{2}+}} 
			\|z_{-3n}\|_{X^{0,\frac{1}{2}-}}.
		\end{split}
	\end{equation}
	Now, using H\"older's inequality we arrive at the required estimate as in the {\bf Subcase D1.1}.\\
	
	Now, consider the case when $|\xi_1-\xi_3|\gg 1$. Let $\xi_1\in I_n$, $\xi_3\in I_m$, then $\xi_2\in I_{-2n-m+k}$, $j,k=O(1)$.
	\\ 
	By Proposition \ref{O(1)} we have than $|\xi_1-\xi_3|\sim |n-m|\gg 1$ and $|n+m|\sim |\xi_1+\xi_3|=|\xi_1|+|\xi_3|\sim |n|$. Considering only the diagonal case $j=k=0$. Using Bernstein's inequality (\ref{Desigual_Bernstein_1}), Corollary \ref{CorolPPi1} and (\ref{Subcase 4.2}), we obtain
	\begin{equation*}
	\begin{split}
		\sup_{\begin{smallmatrix}
					k\in \Z\\|k| \gg  1
			\end{smallmatrix}} &\sum_{\begin{smallmatrix}
				m,n\in\mathbb{Z}\\|m|\sim |n|\sim|k|
		\end{smallmatrix}} |n|^{s+1}\left|\int_{\mathbb{R}^2}
		(w_1)_n (w_2)_n v_m z_{-2n-m}\, dxdt\right| \\
		 &\lesssim  \sup_{\begin{smallmatrix}
					k\in \Z\\|k| \gg  1
			\end{smallmatrix}} \sum_{\begin{smallmatrix}
				m,n\in\mathbb{Z}\\|m|\sim |n|\sim |k|
		\end{smallmatrix}} |n|^{s+1} \|(w_1)_n v_m\|_{L^2_{x,t}}
		\|(w_2)_n z_{-2n-m}\|_{L^2_{x,t}}\\
			& \lesssim \sup_{\begin{smallmatrix}
					k\in \Z\\|k| \gg  1
			\end{smallmatrix}}\sum_{\begin{smallmatrix}
					m,n\in\mathbb{Z}\\|m|\sim |n|\sim |k|\\|m\pm n|\gg 1
			\end{smallmatrix}} \frac{|n|^{-2s-\frac{1}{2}+}}{\sqrt{|m-n||m+n|}}
			\|(w_1)_n\|_{X^{s,\frac{1}{2}+}_\alpha}
			\|v_m\|_{X^{s,\frac{1}{2}+}}
			\|(w_2)_n\|_{X^{s,\frac{1}{2}+}_\alpha}
			\|z_{-m-2n}\|_{X^{0,\frac{1}{2}-}}\\
			& \lesssim 
			\|w_1\|_{X^{s,\frac{1}{2}+}_{\infty,\alpha}}
			\|z\|_{X^{s,\frac{1}{2}-}_\infty}
			\sup_{\begin{smallmatrix}
					k\in \Z\\|k| \gg  1
			\end{smallmatrix}}\sum_{\begin{smallmatrix}
					m,n\in\mathbb{Z}\\|m|\sim |n|\sim |k|\\|m- n|\gg 1
			\end{smallmatrix}} \frac{1}{|m-n|\langle n\rangle^{1/2+2s-}}  
			\|(w_2)_n\|_{X^{s,\frac{1}{2}+}_\alpha}          
			\|v_m\|_{X^{s,\frac{1}{2}+}}\\
			& =:
			\|w_1\|_{X^{s,\frac{1}{2}+}_{\infty,\alpha}}
			\|z\|_{X^{s,\frac{1}{2}-}_\infty}
			\sup_{\begin{smallmatrix}
					k\in \Z\\|k| \gg  1
			\end{smallmatrix}}\mathcal{Y}(k)
		\end{split}
	\end{equation*}
	Let $a_n=\|(w_2)_n\|_{X^{s,\frac{1}{2}+}_\alpha} $ and $b_m=\|v_m\|_{X^{s,\frac{1}{2}+}}$. Then, we have
	\begin{equation*}
	\begin{split}
	\mathcal{Y}(k)&=\sum_{\begin{smallmatrix}
					m,n\in\mathbb{Z}\\|m|\sim |n|\sim |k|\\|m-n|\gg 1
			\end{smallmatrix}} \frac{1}{|m-n|\langle n\rangle^{1/2+2s-}}  
			a_n b_m\\
			& \lesssim \|b_m\|_{\ell^p}\left\| \sum_{|n|\sim|k|} \frac{a_n}{|n|^{1/2+2s-}|m-n|} \right\|_{\ell^{p'}(|m|\sim |k|)}\\
		& \lesssim \|b_m\|_{\ell^p}\left\| \frac{a_n}{|n|^{1/2+2s-}}\right\|_{\ell^{r_1}(|n|\sim |k|)}\left\| \frac{1}{|m|}\right\|_{\ell^{r_2}(|m|\sim |k|)}\\	
		\\
		& \lesssim \|a_n\|_{\ell^p}\|b_m\|_{\ell^p} \frac{1}{|k|^{(\frac12+2s-)-\frac{1}{r_3}}}\frac{1}{|k|^{1-\frac{1}{r_2}}},
	\end{split}
	\end{equation*}
	where 
	\begin{equation}\label{indices2}
	\frac{1}{p'}+1=\frac{1}{r_1}+\frac{1}{r_2}, \quad \textrm{and} \quad \frac{1}{r_1}=\frac{1}{p}+\frac{1}{r_3}
	\end{equation}
	The relations in \eqref{indices2} and the relation $s>\frac14-\frac1{p}$ imply that
	$$
	{(\frac12+2s-)+1-\frac1{r_2}-\frac1{r_3}}>0.
	$$
	
	\noindent
	\textbf{Sub-subcase D2.2.2.} $\sigma=\sigma_{\max}$ \textbf{and} $|\xi_i-\xi_j|\gg 1$ \textbf{for all} $i\neq j$. Let $\xi_j\in I_{n_j}$, $j=1,2,3$ and $\xi\in I_n$.\\
	By Proposition \ref{O(1)}, we have that $|n_i-n_j|\sim |\xi_i-\xi_j|\gg 1$ and also $|n_i+n_j|\sim |\xi_i+\xi_j|\sim |\xi_i|+|\xi_j|\sim |n_i|+|n_j|\sim |\xi_{\max}|\gg 1$.
	\\
	
	Using Bernstein's inequality (\ref{Desigual_Bernstein_1}), Corollary \ref{CorolPPi1}, and (\ref{Subcase 4.2}), we obtain
	\begin{equation*}
		\begin{split}
			I & := \sum_{n_1+n_2+n_3+n=O(1)} |n|^{s+1}\left|\int_{\mathbb{R}^2}(w_1)_{n_1} (w_2)_{n_2} v_{n_3} z_n\, dxdt\right|\\
			& \leq \sum_{n_1+n_2+n_3+n=O(1)} |n|^{s+1} \|(w_1)_{n_1} (w_2)_{n_2}\|_{L^2_{x,t}} \|v_{n_3} z_n\|_{L^2_{x,t}}\\
			& \lesssim \sum_{n_1+n_2+n_3+n=O(1)}
			|n|^{s+1} \|(w_1)_{n_1} (w_2)_{n_2}\|_{L^2_{x,t}}
			\|v_{n_3}\|_{L^\infty_t L^2_x} \|z_n\|_{L^2_{x,t}}\\
			& \lesssim \sum_{n_1+n_2+n_3+n=O(1)}
			|n|^{s+1} \frac{|n|^{-\frac{3}{2}+}}{\sqrt{|n_1+n_2||n_1-n_2|}} \|(w_1)_{n_1}\|_{X^{0,\frac{1}{2}+}_\alpha} \|(w_2)_{n_2}\|_{X^{0,\frac{1}{2}+}_\alpha}
			\|v_{n_3}\|_{X^{0,\frac{1}{2}+}}\\
			& \hspace{0.5 cm} \times \|z_n\|_{X^{0,\frac{1}{2}-}}.
		\end{split}
	\end{equation*}
	We observe the following $|n_1+n_2|\sim |\xi_1+\xi_2|\sim |\xi_3+\xi|\sim |n_3+n|\lesssim |n|$. Thus
	\begin{equation*}    
		\frac{|n|^{-2s-\frac{1}{2}+}}{\sqrt{|n_1+n_2|}}\sim 
		\frac{|n|^{-2s+}}{|n|^\frac{1}{2}\sqrt{|n_3+n|}}\lesssim
		\frac{|n|^{-2s+}}{|n_3+n|^{\frac{1}{2}+\frac{1}{2}}}\sim
		\frac{\langle n_1\rangle^{-2s+} \langle n\rangle^{0-}}{|n_3+n|}.
	\end{equation*}
	From Proposition \ref{O(1)}, we have that $|n_1-n_2|\sim |\xi_1-\xi_2|=|2\xi_1+\xi_3+\xi|=|2n_1+n_3+n+O(1)|\sim |2n_1+n_3+n|\gg 1$. Now, making the change of variable $n_2=O(1)-n_1-n_3-n=-n_1+\kappa$, $\kappa\in\mathbb{Z}$, we get
	\begin{equation*}
		\begin{split}
			I
			&\lesssim \sum_{n_1,n_3,n\in\mathbb{Z}}
			\frac{\langle n_1\rangle^{-2s+} \langle n\rangle^{0-}}{|n_3+n|\sqrt{|2n_1+n_3+n|}} 
			\|(w_1)_{n_1}\|_{X^{s,\frac{1}{2}+}_\alpha}
			\|(w_2)_{-n_1+\kappa}\|_{X^{s,\frac{1}{2}+}_\alpha}
			\|v_{n_3}\|_{X^{s,\frac{1}{2}+}}
			\|z_n\|_{X^{0,\frac{1}{2}-}}\\
			& \lesssim \sup_{n_3,n\in \mathbb{Z}}\left(\sum_{n_1\in\mathbb{Z}}
			\frac{\langle n_1\rangle^{-2s+} \|(w_1)_{n_1}\|_{X^{s,\frac{1}{2}+}_\alpha} \|(w_2)_{-n_1+\kappa}\|_{X^{s,\frac{1}{2}+}_\alpha}}{\sqrt{|2n_1+n_3+n|}}\right) \sum_{n_3,n\in \mathbb{Z}}\frac{\|v_{n_3}\|_{X^{s,\frac{1}{2}+}}
				\|z_n\|_{X^{0,\frac{1}{2}-}}}{|n_3+n|\langle n\rangle^{0+}}.
		\end{split}
	\end{equation*}
	Finally, applying Proposition \ref{sumatorio lp_lp'},  the same argument as in \eqref{x1sub subcase 4.1.2xw}, \eqref{x1sub subcase 4.1.2xw3} and \eqref{x1sub subcase 4.1.2xw4} noting that $(2s+\frac12)\frac{p}{p-2}>1$,  $p>2$ which is true for $s>\frac{1}{4}-\frac{1}{p}$, we obtain the required estimate. The case $p=2$ is similar as in the previous case.

	Finally, gathering information from all four cases, we conclude that the estimate (\ref{trilinear estimate_1}) holds.
\end{proof}

\begin{remark}
{\bf 1)}  The regularity restriction $s>\frac14-\frac1{p}$ is necessary only in the resonant {\bf Sub-cases D1.5} and {\bf  D2.2}. In all other cases we get a better regularity  requirement, viz.,  $s>\frac14-\frac3{2p}$. In the particular case $p=2$, the later regularity requirement coincides with the sharp result $s>-\frac12$ obtained in the authors previous work \cite{Carvajal_and_Panthee}. 
\\
{\bf 2)} 
	In the cases $p=2, 3, 4$, it  suffices to prove the trilinear estimate  \eqref{trilinear estimate_1} only for $\frac14-\frac{1}{p}< s\leq 0$. Since, for $s\geq 0$ and $\xi+\xi_1+\xi_2+\xi_3=0$, one has $\langle\xi\rangle^s\lesssim \langle\xi_1\rangle^s \langle\xi_2\rangle^s \langle\xi_3\rangle^s$. The estimate for $s>0$ follows easily from the case $s=0$ using the above triangular inequality. 
\end{remark}



\section{Proof of the Main Result}

This section is dedicated to presenting the proof of the main result of this work, stated in Theorem \ref{maintheo}.
\begin{proof}[Proof of Theorem \ref{maintheo}]
	Let $(v_0,w_0)\in M^{2,p}_s(\mathbb{R})\times M^{2,p}_s(\mathbb{R})$ with $s>\frac14-\frac1{p}$, $2\leq p <\infty$ and $T\leq 1$.
	Let us define an application
	\begin{equation}\label{Def.Operador Phi}
		\left\{
		\begin{aligned}
			\Phi(v,w)(t) &
			:=\eta_1(t)S(t)v_0+\eta_T(t)\int_0^t S(t-t')\partial_x(v w^2)(t') dt',\\
			\Psi(v,w)(t) &
			:=\eta_1(t)S_\alpha(t)w_0+\eta_T(t)\int_0^t S_\alpha(t-t')\partial_x(v^2 w)(t') dt'.
		\end{aligned}
		\right.
	\end{equation}
	
	We will prove that there exists a time $T>0$ such that the application $\Phi\times \Psi$ defines a contraction on the ball
	\begin{equation}\label{Def Omega de a}
		\Omega_a:=\{(u,v)\in X^{s,b}_p\times X^{s,b}_{p,\alpha}: \|(v,w)\|_{X^{s,b}_p\times X^{s,b}_{p,\alpha}}\leq a \},
	\end{equation}
	for suitably chosen radius $a>0$.
	
	We start by proving  that $\Phi\times \Psi:\Omega_a\rightarrow \Omega_a$ 
	is well defined. Let $(v,w)\in \Omega_a$. We start by estimating the first component $\Phi$. Using the definition (\ref{Def.Operador Phi}), Lemma \ref{Lema 2.1}, Lemma \ref{Lema 2.2}, Proposition \ref{estimativ bilinear u1u2} and the definition (\ref{Def Omega de a}), we obtain
	\begin{equation}\label{1}
		\begin{split}
			\|\Phi (v,w)\|_{X^{s,b}_p} 
			& \leq \|\eta_1(t)S(t)v_0\|_{X^{s,b}_p}+\|\eta_T
			(t)\int_0^t S(t-t')\partial_x(v w^2) dt'\|_{X^{s,b}_p}\\
			& \leq c_0\|v_0\|_{M^{2,p}_s}+c_1 T^{1+b'-b}\|\partial_x(v w^2)\|_{X^{s,b'}_p}\\
			& \leq c_0\|v_0\|_{M^{2,p}_s}+c_1 T^{1+b'-b}
			c_2 \|v\|_{X^{s,b}_p}
			\|w\|^2_{X^{s,b}_{p,\alpha}}\\          
			& \leq c\|(v_0,w_0)\|_{M^{2,p}_s\times M^{2,p}_s} + c T^{1+b'-b}\|(v,w)\|^3_{X^{s,b}_p\times X^{s,b}_{p,\alpha}}\\
			& \leq c\|(v_0,w_0)\|_{M^{2,p}_s\times M^{2,p}_s} + c T^{1+b'-b} a^3.
		\end{split}
	\end{equation}
	Similarly, it is not difficult to get
	\begin{equation}\label{2}
		\|\Psi (v,w)\|_{X^{s,b}_{p,\alpha}} 
		\leq c\|(v_0,w_0)\|_{M^{2,p}_s\times M^{2,p}_s}+c T^{1+b'-b} a^3.
	\end{equation}
	Therefore, from (\ref{1}) and (\ref{2}), we have
	\begin{equation}\label{ppn1}
		\|(\Phi \times \Psi)(v,w)\|_{X^{s,b}_p\times X^{s,b}_{p,\alpha}} 
		\leq c\|(v_0,w_0)\|_{M^{2,p}_s\times M^{2,p}_s}+c T^{1+b'-b} a^3.
	\end{equation}
	
	If we choose $a=2c\|(v_0,w_0)\|_{M^{2,p}_s\times M^{2,p}_s}$ and $T>0$ such that $cT^{1+b'-b}a^2 < \frac{1}{2}$, the estimate \eqref{ppn1} yields
	\begin{equation}\label{2_5}
		\|(\Phi \times \Psi)(v,w)\|_{X^{s,b}_p\times X^{s,b}_{p,\alpha}} 
		\leq \frac{a}{2} + \frac{a}{2}.
	\end{equation}
	Thus, $\Phi\times \Psi (v,w)\in\Omega_a$.
	
	Now, we move to prove that $\Phi\times \Psi$ is a contraction.
	Given $(v,w),(\Tilde{v},\Tilde{w})\in\Omega_a$, with the argument used above, one has
	\begin{equation}\label{3}
		\begin{split}
			\|\Phi\times \Psi (v,w)- \Phi\times \Psi (\Tilde{v},\Tilde{w})\|_{X^{s,b}_p\times X^{s,b}_{p,\alpha}} 
			& \leq cT^{1+b'-b}
			(\|(v,w)\|^2_{X^{s,b}_p\times X^{s,b}_{p,\alpha}}+
			\|(\Tilde{v},\Tilde{w})\|^2_{X^{s,b}_p\times X^{s,b}_{p,\alpha}})\\
			& \quad\times \|(v,w)-(\Tilde{v},\Tilde{w})\|_{X^{s,b}_p\times X^{s,b}_{p,\alpha}}\\
			& \leq 
			2c T^{1+b'-b}a^2\|(v,w)-(\Tilde{v},\Tilde{w})\|_{X^{s,b}_p\times X^{s,b}_{p,\alpha}}.         
		\end{split}
	\end{equation}
 For the choice of $a$ and $T$ in \eqref{2_5}, if we choose $T$ also satisfying
	\begin{equation*}
		2c  T^{1+b'-b} a^2<\frac{1}{2},
	\end{equation*}
the estimate \eqref{3} shows that $\Phi\times\Psi$ is a contraction of $\Omega_a$ into $\Omega_a$. Hence, there exists a unique fixed point $(v,w)$ that solves  the integral equation (\ref{Def.Operador Phi}) for $t\in [0,T]$.	
	The rest of the proof follows a standard argument, so we omit the details.
\end{proof}


\vskip 0.3cm
\noindent{\bf Acknowledgements.} 
 The first author acknowledges the support from CNPq Universal project  (\#406460/2023-0).
The third author acknowledges the grant from FAPESP, Brazil (\#2024/10613-4). \\


\noindent
{\bf Conflict of interest statement.} 
On behalf of all authors, the corresponding author states that there is no conflict of interest.\\

\noindent 
{\bf Data availability statement.} 
The datasets generated during and/or analysed during the current study are available from the corresponding author on reasonable request.



\end{document}